\documentclass{amsart}
\RequirePackage{luatex85} 

\usepackage{graphicx,xcolor}

\usepackage[cmtip,arrow,all]{xy}

\usepackage{lscape}
\usepackage{bxeepic}

\usepackage{arydshln}

\usepackage[bookmarks=false]{hyperref}

\usepackage{aliascnt}
\newcommand*{\autoreftheorem}[2]{%
  \newaliascnt{#1}{theorem}%
  \newtheorem{#1}[#1]{#2}%
  \aliascntresetthe{#1}%
  \expandafter\newcommand\csname #1autorefname\endcsname{#2}%
}

\theoremstyle{plain}
\autoreftheorem{thm}{Theorem}
\autoreftheorem{lem}{Lemma}
\autoreftheorem{prop}{Proposition}
\autoreftheorem{cor}{Corollary}

\theoremstyle{definition}
\autoreftheorem{defin}{Definition}
\autoreftheorem{rem}{Remark}
\autoreftheorem{ex}{Example}

\newcommand{\Frac}[2]{\displaystyle{\frac{#1}{#2}}}

\newcommand{\Ps}{Poincar\'e series}
\newcommand{\bbeta}{\widehat{\beta}}
\newcommand{\ggamma}{\widehat{\gamma}}
\newcommand{\llambda}{\widehat{\lambda}}
\newcommand{\ttheta}{\widehat{\theta}}
\newcommand{\uu}{\widehat{u}}
\newcommand{\vv}{\widehat{v}}
\newcommand{\wrr}{\widehat{r}}
\newcommand{\wt}{\widehat{t}}
\newcommand{\wl}{\widehat{\ell}}
\newcommand{\hh}{\widehat{h}}
\newcommand{\bb}{\widehat{b}}

\newcommand{\Tm}[1]{T_{m}^{(#1)}}
\newcommand{\Tmbar}[1]{\overline{T}_{m}^{(#1)}}
\newcommand{\ie}{i.e.,}
\newcommand{\CEE}[3]{\Ext_{G(m+1)}^{#1}(\Tmbar{#2} \otimes #3)}
\newcommand{\specialmassey}[2]{\mu_{#1}(#2)}

\renewcommand{\SS}{spectral sequence} 
\newcommand{\DGA}{differential graded algebra}
\newcommand{\SES}{short exact sequence}
\newcommand{\LES}{long exact sequence}
\newcommand{\ANSS}{Adams-Novikov spectral sequence}
\newcommand{\CESS}{Cartan-Eilenberg spectral sequence}
\newcommand{\CSS}{chromatic spectral sequence}
\newcommand{\RSS}{resolution spectral sequence}
\newcommand{\SDSS}{small descent spectral sequence}

\newcommand{\coker}{\operatorname{coker}}
\newcommand{\im}{\operatorname{im}}
\newcommand{\Ext}{\operatorname{Ext}}

\newcommand{\ints}{\mathbb{Z}}
\def\nequiv{{{\equiv} \kern -.6em {/\:} }}

\begin{document}
\title[The Method of Infinite Descent II]{The Method of Infinite Descent in Stable Homotopy Theory II}

\author[Hirofumi Nakai]{Hirofumi Nakai}
\address{Department of Mathematics, Tokyo City University, Tokyo, Japan} 
\email{hnakai@tcu.ac.jp}

\author[Douglas C. Ravenel]{Douglas C. Ravenel}
\address{Department of Mathematics, University of Rochester, Rochester, NY} 
\email{doug@math.rochester.edu}

\begin{abstract}
This paper is a continuation of the version I of the same title,   
which intends to clarify and expand the results in the last chapter of ``the green book'' by the second author.  
In particular, we give the stable homotopy groups of $p$-local spectra $T(m)_{(1)}$ for $m>0$.  
This is a part of a program to compute the $p$-components of $\pi_{*}(S^{0})$ 
through dimension $2p^{4}(p-1)$ for $p>2$.  
We will refer to the results from the version I freely as if they 
were in the first four sections of this paper, 
which begins with section 5.
\end{abstract}

\maketitle

\setcounter{section}{4}

\section{Introduction}

In \cite{Rav:MU} the second author described a
method for computing the Adams-Novikov $E_{2}$-term for spheres 
and used it to determine the stable homotopy groups 
through dimension 108 for $p=3$ and 999 for $p=5$.  
The latter computation was a substantial
improvement over prior knowledge, and neither has been improved upon since.  
It is generally agreed among homotopy theorists that it is not
worthwhile to try to improve our knowledge of stable homotopy groups
by a few stems, but that the prospect of increasing the known range by
a factor of $p$ would be worth pursuing.  
This possibility may be
within reach now, due to a better understanding of the methods of
\cite[Chapter 7]{Rav:MU} and improved computer technology.  
This paper should be regarded as laying the foundation for a program to compute
$\pi_{*}(S^{0})_{(p)}$ through roughly dimension $2p^{4}(p-1)$,
{\ie} 324 for $p=3$ and 5,000 for $p=5$.

\subsection{Summary of I}

First we review the version I \cite{Rav:Ch7I} briefly.  
The method referred to in the title involves the connective $p$-local
ring spectra $T(m)$ satisfying
\[
BP_{*}(T(m)) = BP_{*}[t_{1}, \dots, t_{m}] \subset BP_{*} (BP) 
\]
and the natural map $T(m)\to BP$ which is an equivalence below dimension $|t_{m+1}|$. 
In particular, we have 
$T(0) = S^{0}_{(p)}$
and 
$T(\infty) = BP$.

For a Hopf algebroid $(A,\Gamma)$ and $\Gamma$-comodule $M$, 
we will often drop the first variable of Ext for short, 
{\ie} 
$\Ext_{\Gamma}(A,M)$
will be denoted by $\Ext_{\Gamma}(M)$.  
If we define the quotient module $\Gamma(k)$ by 
\[
\Gamma(k) 
= BP_{*}(BP)/(t_{1}, \ldots, t_{k-1}) 
\cong BP_{*}[t_{k}, t_{k+1}, \ldots]
\]
then the pair $(BP_{*},\Gamma(k))$ forms a Hopf algebroid, 
whose structure maps are inherited from $(BP_{*},BP_{*}(BP))$.  
Note that $\Gamma(1) = BP_{*}(BP)$.  
By the change-of-rings isomorphism \cite[Theorem A1.3.12]{Rav:MU2nd}, 
the Adams-Novikov $E_{2}$-term for $T(m)$
is reduced to $\Ext_{\Gamma(m+1)}^{*}(BP_{*})$.

It is not difficult to find the structure of
$\Ext_{\Gamma(m+1)}^{*}(BP_{*})$ in low dimensions.  
For example, it is known below dimension $p|v_{m}|$ 
and below dimension $p|v_{m+1}|$ \cite[Theorem 7.1.9 and 7.1.13]{Rav:MU}.
Moreover, the structure of $\Ext_{\Gamma(m+1)}^{*}(BP_{*})$ below
dimension $p^{2}|v_{m+1}|$ was determined in
Theorem 4.5.

We have constructed the {\SES} of $\Gamma(m+1)$-comodules
\setcounter{equation}{\value{thm}}
\begin{equation}\label{ses-DE-1}
0 \longrightarrow 
    BP_{*} \stackrel{i_{1}}{\longrightarrow} 
       D_{m+1}^{0} \stackrel{j_{1}}{\longrightarrow} 
          E_{m+1}^{1} \longrightarrow 0 
\qquad
\mbox{for $m \ge 0$}
\end{equation}
\setcounter{thm}{\value{equation}}%
where the map $i_{1}$ induces an isomorphism of $\Ext^{0}$ 
(cf.~Theorem 3.7), 
and $D_{m+1}$ is a weak injective $\Gamma(m+1)$-comodule.  
So, we have isomorphisms  
\[
\Ext_{\Gamma(m+1)}^{t}(E_{m+1}^{1}) \cong \Ext_{\Gamma(m+1)}^{t+1}(BP_{*})
\quad
\mbox{for $t \ge 0$.} 
\]
We may also assume the existence of the {\SES} 
\setcounter{equation}{\value{thm}}
\begin{equation}\label{ses-DE-2}
0 \longrightarrow E_{m+1}^{1} \stackrel{i_{2}}{\longrightarrow} D_{m+1}^{1} \stackrel{j_{2}}{\longrightarrow} E_{m+1}^{2} \longrightarrow 0. 
\end{equation}
\setcounter{thm}{\value{equation}}%
where $D_{m+1}^{1}$ is weak injective: 
it is specifically constructed in Lemma 4.1 for $m=0$ and odd $p$
with the map $i_{2}$ inducing an isomorphism in $\Ext^{0}$. 
For $m>0$, it is shown that 
$v_{1}^{-1}E_{m+1}^{1}$ is weak injective with 
\[
\Ext_{\Gamma(m+1)}^{0}(v_{1}^{-1}E_{m+1}^{1}) \cong v_{1}^{-1}\Ext_{\Gamma(m+1)}^{1}(BP_{*})
\]
thus we may regard $D_{m+1}^{1}$ as $v_{1}^{-1}E_{m+1}^{1}$ at worst 
(cf.~Lemma 3.18).  

Of course, it is desirable to define $D_{m+1}^{1}$ for $m>0$ to make its $\Ext^{0}$ as small as possible.
If we assume that the map $i_{2}$ for induces an isomorphism in $\Ext^{0}$, 
then we have isomorphisms 
\[
\Ext_{\Gamma(m+1)}^{t}(E_{m+1}^{2})
\cong 
\Ext_{\Gamma(m+1)}^{t+2}(BP_{*}) 
\quad
\mbox{for $t \ge 0$.}
\]
We constructed such isomorphisms\footnote{
Unfortunately, $i_{2}$ induces an isomorphism in $\Ext^{0}$
only below dimension $p|v_{m+2}|$ for $m>0$. 
See \autoref{failure-5.2}.
}
and computed the Ext groups below dimension $p^{2}|v_{m+1}|$ 
by producing $E_{m+1}^{2}$ satisfying some desirable conditions 
and the weak injective $D_{m+1}^{1}$ as the induced extension 
(cf.~Corollary 4.3):
\[
\xymatrix{
0 \ar[r] & E_{m+1}^{1} \ar[r]^{i_{2}} \ar@{=}[d] & D_{m+1}^{1} \ar[r]^{j_{2}} \ar@{^{(}->}[d] & E_{m+1}^{2} \ar[r] \ar@{^{(}->}[d] & 0  \\
0 \ar[r] & E_{m+1}^{1} \ar[r] & v_{1}^{-1}E_{m+1}^{1} \ar[r] & E_{m+1}^{1}/(v_{1}^{\infty}) \ar[r] & 0. 
}
\]
Since there is no Adams-Novikov differential and no nontrivial group extension in this range 
(except in the case $m=0$ and $p=2$), 
this also determines $\pi_{*}(T(m))$ in the same range.  
This was the goal of the version I.

\subsection{Introduction to II}

To descend from $T(m+1)$ to $T(m)$, 
we can consider some interpolating spectra $T(m)_{(i)}$
introduced in 
Lemma 1.15. 
Each $T(m)_{(i)}$ is the $T(m)$-module spectrum satisfying 
\[
BP_{*}(T(m)_{(i)}) = BP_{*} (T(m))\{t_{m+1}^{\ell} \mid 0 \le \ell < p^{i} \} 
\]
and the natural equivalence $T(m)_{(i)} \to T(m+1)$ below dimension $p^{i}|t_{m+1}|$. 
In particular, we have 
$T(m)_{(0)} = T(m)$
and 
$T(m)_{(\infty)} = T(m+1)$.

The Adams-Novikov $E_{2}$-term for $T(m)_{(i)}$ is 
\[
E_{2}^{n,*} = 
\Ext_{BP_{*}(BP)}^{n,*}(BP_{*},BP_{*}(\Tm{i}))
\]
and it is reduced to $\Ext_{\Gamma(m+1)}^{n,*}(BP_{*}, \Tm{i})$ by
Lemma 1.15, 
where $\Tm{i}$ is the $BP_{*}$-module generated by
$\{ t_{m+1}^{\ell} \mid 0 \le \ell < p^i \}$.
Then, we have the $3$-term resolution
of $\Tm{i}$ by tensoring the {\SES} \eqref{ses-DE-1} with $\Tm{i}$,
and the associated {\SS} $\{ E_{r}^{n,t}, d_{r} \}_{r \ge 1}$
converging to $\Ext_{\Gamma(m+1)}^{*}(\Tm{i})$ with
\setcounter{equation}{\value{thm}}
\begin{equation}\label{rss-E1}
E_{1}^{n,t}
= 
\begin{cases}
\Ext_{\Gamma (m+1)}^{0}(\Tm{i} \otimes_{BP_{*}} D_{m+1}^{0})
             &\mbox{for $(n,t)= (0,0)$,}    \\
\Ext_{\Gamma (m+1)}^{t}(\Tm{i} \otimes_{BP_{*}} E_{m+1}^{1})
             &\mbox{for $n=1$,}             \\
0            &
\mbox{otherwise.}
\end{cases}
\end{equation}
\setcounter{thm}{\value{equation}}%
The only nontrivial differential is 
$d_{1}: E_{1}^{0,0} \to E_{1}^{1,0}$
induced by $j_{1}$ \eqref{ses-DE-1}, 
and the {\SS} collapses from $E_{2}$-term.  
Thus we have

\begin{prop}\label{prop-ANSS-Tmi}
The Adams-Novikov $E_{2}$-term for $T(m)_{(i)}$ is 
\[
\Ext_{\Gamma(m+1)}^{n}(\Tm{i})
\cong 
\begin{cases}
\ker d_{1}
 & \mbox{for $n=0$},  \\
\coker d_{1}
 & \mbox{for $n=1$},  \\
\Ext_{\Gamma (m+1)}^{n-1}(\Tm{i} \otimes_{BP_{*}}E_{m+1}^{1})
 & \mbox{for $n \ge 2$.}  
\end{cases}
\]
Note that the groups for $n=0$ and $1$ were determined 
in \cite[Proposition 2.5, Theorem 4.1 and \S 5]{NR:GeneralizedImageJ}
(See also \autoref{prop-H(C)}).
\end{prop}

Once we know about $T(m)_{(i+1)}$ for some $i$, 
we can descend the value of $i$ by using the {\SDSS} 
(Theorem 1.21), 
whose $E_{1}$-term is 
\[
E(h_{m+1,i}) \otimes P(b_{m+1,i}) \otimes \pi_{*}(T(m)_{(i+1)})
\]
where 
$h_{m+1,i} \in E_{1}^{1,2p^{i}(p^{m+1}-1)}$
and 
$b_{m+1,i} \in E_{1}^{2,2p^{i+1}(p^{m+1}-1)}$
are permanent cycyles.  
Note that we know $\pi_{*}(T(1)_{(3)})$ below dimension $p^3|t_{2}|$ 
by 
Theorem 4.5 
without any use of {\SS}s, 
since the dimension is smaller than $p^2|t_{3}|$
and $T(1)_{(3)} = T(2)$ in that range.  
This allows us to compute $\pi_{*}(T(1))$ 
from the information of $\pi_{*}(T(1)_{(3)})$.  
Since $T(0)_{(4)} = T(1)$ below dimension $p^{4}|v_{1}|$, 
this also makes possible to have $\pi_{*}(S^{0})$ in the same range.

\bigskip

In this paper we assume that $m>0$ unless otherwise noted.  
The main results are the determination of the Adams-Novikov $E_{2}$-terms for 
$T(m)_{(1)}$ below dimension $p|v_{m+3}|$
in \autoref{CESS-Einfty}.  
In this range there is still no room for Adams-Novikov differentials, so
the homotopy and Ext calculations coincide\footnote{For $m=0$, the second author determined the structure of
$\Ext_{\Gamma(1)}^{*+2}(T_{0}^{(1)})$ in \cite[Theorem 7.5.1]{Rav:MU2nd} for
$p>2$ below dimension $(p^{3}+p)|v_{1}|$.}.  
It is only when we pass from $T(m)_{(1)}$ to $T(m)$ 
that we encounter Adams-Novikov
differentials below dimension $p^{2}|v_{m+2}|$.  For $m=0$, the first
of these is the Toda differential 
$d_{2p-1}(\beta_{p/p})=\alpha_{1}\beta_{1}^{p}$ 
of \cite{Tod:Rel} and \cite{Tod:XPow}, 
and the relevant calculations were the subject of
\cite[Chapter 7]{Rav:MU2nd}.  
An analogous differential for $m>0$ was also established in \cite{Rav:First}, and we will discuss it somewhere else
in the future.

\section{A variant of {\CESS}}\label{sec-E2'}

Assume that $M$ is a $\Gamma(k)$-comodule for some $k$.  
Once we know the structure of $\Ext_{\Gamma(k)}^{*}(M)$, 
there is an inductive step reducing the value of $k$.  
Set 
\[
A(k) = \ints_{(p)}[v_{1}, \ldots, v_{k}] 
\quad
\mbox{and}
\quad
G(k) = A(k)[t_{k}].
\]  
The pair $(A(k),G(k))$ is a Hopf algebroid.  
Then we have an extension of Hopf algebroids 
(cf.~Proposition 1.2)
\[
(A(m),G(m))
\longrightarrow 
(BP_{*},\Gamma(m))
\longrightarrow 
(BP_{*},\Gamma(m+1))
\]
and the associated {\CESS}
\[
\Ext_{G(m)}^{*}(\Ext_{\Gamma(m+1)}^{*}(M))
\quad
\Longrightarrow
\quad
\Ext_{\Gamma(m)}^{*}(M).
\]
A $\Gamma(m+1)$-comodule $M$ is naturally a $\Gamma(m+2)$-comodule, 
and we will denote $\Ext_{\Gamma(m+2)}^{0}(M)$ by $\overline{M}$ for short.  
In particular, we have 
\[
\Tmbar{i} = 
A(m+1)\{ t_{m+1}^{\ell} \mid 0 \le \ell < p^{i} \}. 
\] 
Then the Cartan-Eilenberg $E_{2}$-terms converging to 
$\Ext_{\Gamma(m+1)}^{*}(\Tm{i} \otimes_{BP_{*}} E_{m+1}^{1})$
is 
\setcounter{equation}{\value{thm}}
\begin{align}
\tilde{E}_{2}^{s,t}
& = \Ext_{G(m+1)}^{s}(\Ext_{\Gamma(m+2)}^{t}(\Tm{i} \otimes_{BP_{*}} E_{m+1}^{1})) \nonumber\\
& \cong \Ext_{G(m+1)}^{s}(\Tmbar{i} \otimes_{A(m+1)} \Ext_{\Gamma(m+2)}^{t}(E_{m+1}^{1})) 
\label{CESS-for-E_{m+1}^{1}}
\end{align}
\setcounter{thm}{\value{equation}}%
with differentials 
$\tilde{d}_{r}: \tilde{E}_{r}^{s,t} \to \tilde{E}_{r}^{s+r,t-r+1}$.  
Since the case $s=t=0$ is not interesting, we may consider only for $s+t \ge 1$.  

For simplicity, 
we will hereafter omit the subscript in $\otimes_{A(m+1)}$, 
and we will denote $\Ext_{\Gamma(m+2)}^{t}(BP_{*})$ by $U_{m+1}^{t}$. 
Since $D_{m+1}^{0}$ is weak injective, we have isomorphisms 
$\Ext_{\Gamma(m+2)}^{t}(E_{m+1}^{1}) \cong U_{m+1}^{t+1}$
and 
\setcounter{equation}{\value{thm}}
\begin{equation}\label{cess-E2-t>0}
\tilde{E}_{2}^{s,t} \cong 
\CEE{s}{i}{U_{m+1}^{t+1}} 
\qquad
\mbox{for $t \ge 1$}.
\end{equation}
\setcounter{thm}{\value{equation}}%
Note that the structure of $U_{m+1}^{*}$ can be read from Theorem 4.5.  
This will be discussed again in \autoref{cor-E2again}.

\bigskip

To describe $\tilde{E}_{2}^{s,0}$, 
we need a resolution of 
$\overline{E}_{m+1}^{1} = \Ext_{\Gamma(m+2)}^{0}(E_{m+1}^{1})$.  
The obvious one is obtained by applying 
$\Ext_{\Gamma(m+2)}^{0}(-)$ to \eqref{ses-DE-2}. 
In practice, there is a ``smaller resolution''.

Now we recall some notations used in the version I.  
For a fixed positive integer $m$, 
we will set 
$\vv_{i} = v_{m+i}$ 
and 
$\wt_{i} = t_{m+i}$, 
and define
\[
\bbeta_{i/e_{1},e_{0}}
= \Frac{\vv_{2}^{i}}{p^{e_{0}}v_{1}^{e_{1}}}, 
\qquad
\bbeta_{i/e_{1}}
= \bbeta_{i/e_{1},1},  
\qquad
\bbeta_{i}
= \bbeta_{i/1}, 
\]
\[
\bbeta'_{i/e_{1}}
= \Frac{1}{i} \bbeta_{i/e_{1}}, 
\qquad
\bbeta'_{i}
= \bbeta'_{i/1},  
\qquad\mbox{and } 
\quad
\ggamma_{i}
= \Frac{\vv_{3}^{i}}{pv_{1}v_{2}}.
\]
Then we have

\begin{prop}\label{thm-B1}
Let $B_{m+1}$ be the $A(m+1)$-module 
generated by $\bbeta'_{i/i}$ for $i>0$.  
Then $B_{m+1}$ is a sub $G(m+1)$-comodule of $E_{m+1}^{1}/(v_{1}^{\infty })$
and it is invariant over $\Gamma(m+2)$.  
Its {\Ps} is
\[
g(B_{m+1}) = g_{m+1}(t) \sum_{k\ge0}
\frac{x^{p^{k+1}} (1-y^{p^{k}})}{(1-x^{p^{k+1}})(1-x_{2}^{p^{j}})}
\]

\noindent
where
$y=t^{|v_{1}|}$,
$x=t^{|\vv_{1}|}$,  
$x_{2} = t^{|\vv_{2}|}$ 
and
\[
g_{m+1}(t) = 
\prod_{i=1}^{m+1} \frac{1}{1-y_{i}}
\qquad
\mbox{where $y_{i}=t^{|v_{i}|}$.}
\]
\end{prop}

\begin{proof}
This is \cite[Theorem 2.4]{NR:OnthePrimitiveBeta}.  
To clarify that $\bbeta'_{i/i}$ are in $E_{m+1}^{1}/(v_{1}^{\infty})$, 
note that an element in $N^{2}$ lies in
$E_{m+1}^{1}/ (v_{1}^{\infty })$ if and only if it has trivial image
in $(M^{0}/D_{m+1}^{0})/(v_{1}^{\infty })$.
This can be shown by the following commutative diagram\footnote{
For $m=0$ and $p>2$, $E_{1}^{1}/(v_{1}^{\infty})$ is isomorphic to $N^{2}$.}. 
\[
\xymatrix{
0 \ar[r] & E_{m+1}^{1} \ar[r] \ar@{>->}[d] & v_{1}^{-1}E_{m+1}^{1} \ar[r] \ar@{>->}[d] & E_{m+1}^{1}/(v_{1}^{\infty}) \ar[r] \ar@{>->}[d] & 0 \\
0 \ar[r] & N^{1} \ar[r] \ar@{->>}[d] & M^{1} \ar[r] \ar@{->>}[d] & N^{2} \ar[r] \ar@{->>}[d] & 0  \\
0 \ar[r] & M^{0}/D_{m+1}^{0} \ar[r] & v_{1}^{-1}(M^{0}/D_{m+1}^{0}) \ar[r] & (M^{0}/D_{m+1}^{0})/(v_{1}^{\infty }) \ar[r] & 0
}
\]
where $M^{i}$ and $N^{i}$ are usual chromatic comodules.  
Define $w \in D_{m+1}^{0}$ by  
\setcounter{equation}{\value{thm}}
\begin{equation}\label{eq-w}
w= (1-p^{p-1})\llambda_{1}^{p} - v_{1}^{p^{m+1}-1}\lambda_{1}. 
\end{equation}
\setcounter{thm}{\value{equation}}%
Then we have 
$\vv_{2} = p (\llambda_{2} + \lambda_{1} w)$
and 
\[
\bbeta'_{i/i} 
= \frac{p^{i} (\llambda_{2} + \lambda_{1} w)^{i} }{ip v_{1}^{i}}
= \frac{p^{i-1} (\llambda_{2} + \lambda_{1} w)^{i} }{i v_{1}^{i}} 
\]
which is clearly in $(M^{0}/D_{m+1}^{0})/(v_{1}^{\infty })$ as desired. 
\end{proof}

Let $W_{m+1}$ be the $G(m+1)$-comodule\footnote{
For $m=0$ and $p>2$, 
we may simply set 
$W_{1}= \Ext_{\Gamma(2)}^{0}(D_{1}^{1})$ 
(cf.~\cite[(7.2.17)]{Rav:MU2nd}), 
since the map 
$E_{1}^{1} \to D_{1}^{1}$ 
induces an isomorphism in $\Ext_{\Gamma(1)}^{0}$. 
}
defined by the induced extension
in the following commutative diagram
(cf.~\cite[(1.4)]{NR:OnthePrimitiveBeta}):
\[
\xymatrix{
0 \ar[r] & \overline{E}_{m+1}^{1} \ar[r]^{\iota} \ar@{=}[d] & 
W_{m+1} \ar[r]^{\rho} \ar@{^{(}->}[d] & 
B_{m+1} \ar[r] \ar@{^{(}->}[d] & 0 \\
0 \ar[r] & \overline{E}_{m+1}^{1} \ar[r] & 
v_{1}^{-1} \overline{E}_{m+1}^{1} \ar[r] & 
\overline{E}_{m+1}^{1}/(v_{1}^{\infty}) \ar[r] & 0
}
\]
In fact, we can describe $W_{m+1}$ explcitly.  Recall that
\begin{align*}
\Ext_{\Gamma(m+2)}^{1}(BP_{*})
& \cong A(m+1)
\left\{
\Frac{\vv_{2}^{i}}{ip}
\mid i>0
\right\}.
\end{align*}
Applying $\Ext_{\Gamma(m+2)}$ to \eqref{ses-DE-1} we have the {\SES} 
\[
0 \longrightarrow 
A(m)[\llambda_{1}]/A(m+1) \longrightarrow 
\overline{E}_{m+1}^{1} \stackrel{\delta}{\longrightarrow} 
U_{m+1}^{1}
\longrightarrow
0. 
\]
Then, a lift of 
$\vv_{2}^{i}/ip \in U_{m+1}^{1}$
to $\overline{E}_{m+1}^{1}$ is given by 
\[
b_{i}= 
\Frac{\vv_{2}^{i}-(v_{1}w)^{i}}{ip}
\qquad
\mbox{where $w$ is as in \eqref{eq-w}}.
\]
and a lift of 
$\bbeta '_{i/i} \in B_{m+1}$ 
to $W_{m+1}$ is given by 
\[
v_{1}^{-i} b_{i} = 
\sum_{0<j \le i}
\binom{i-1}{j-1}
\frac{(pv_{1}^{-1}\llambda_{2})^{j}}{pj}
w^{i-j}. 
\] 
So, $W_{m+1}$
is the subcomodule of $M^{1}$ 
obtained by
adjoining 
$v_{1}^{-i} b_{i}$ $(i>0)$ 
to $\overline{E}_{m+1}^{1}$.

\bigskip

The following properties of $W_{m+1}$ can be read from \cite[Theorem 2.4]{NR:OnthePrimitiveBeta}.

\begin{prop}\label{prop-W-weakinj}
$W_{m+1}$ is weak injective and 
the map 
$\iota: \overline{E}_{m+1}^{1} \to W_{m+1}$
induces an isomorphism in $\Ext^{0}$: 
we have 
$\Ext_{G(m+1)}^{0}(W_{m+1}) \cong U_{m+1}^{1}$. 
\end{prop}

Now we have a 3-term resolution of $\overline{E}_{m+1}^{1}$
\[
0 \longrightarrow 
\overline{E}_{m+1}^{1} \stackrel{\iota}{\longrightarrow} 
W_{m+1} \stackrel{\rho}{\longrightarrow} 
B_{m+1} \longrightarrow 0. 
\]
Let $C^{*,s}$ denote the cochain complex obtained by applying 
$\Ext_{G(m+1)}^{s}(\Tm{j}\otimes -)$ 
to the sequence 
\[
\xymatrix{
\overline{D}_{m+1}^{0} \ar[r]^{\iota \circ (j_{1})_{*}} & W_{m+1} \ar[r]^{\rho} & B_{m+1}
}
\]
and let $H^{*,s}(C)$ be the associated cohomology group.  
Then we have 

\begin{prop}\label{prop-H(C)}
For $n=0$ and $1$, $H^{n,0}(C)$ is isomorphic to 
the Adams-Novikov $E_{2}$-term 
$\Ext_{\Gamma(m+1)}^{n}(\Tm{j})$. 
\end{prop}

\begin{proof}
Since $W_{m+1}$ is weak injective over $G(m+1)$, 
$\Tmbar{i} \otimes W_{m+1}$ is also weak injective by 
Lemma 1.14 
and $C^{1,s}=0$ for $s \ge 1$.  
We have the commutative diagram
\[
\xymatrix{
 & C^{0,0} \ar[r] \ar[d]^-{(j_{1})_{*}} & C^{1,0} \ar[r] \ar@{=}[d] & C^{2,0} \ar@{=}[d]    \\
0 \ar[r] & \tilde{E}_{2}^{0,0} \ar[r]^{\iota_{*}} & C^{1,0} \ar[r]^{\rho_{*}} & C^{2,0} \ar[r] & \tilde{E}_{2}^{1,0} \ar[r] & 0 
\quad
\mbox{(exact)}
}
\]
and isomorphisms
$C^{2,s-1} \cong \tilde{E}_{2}^{s,0}$ for $s \ge 2$.  
The map $(j_{1})_{*}$ coincides with the differential 
$d_{1}: E_{1}^{0,0} \to E_{1}^{1,0}$
of the {\RSS} of \eqref{rss-E1}, 
so we have 
\begin{align*}
H^{0,0}(C) 
& = \ker (j_{1})_{*} 
= \ker d_{1}, \\
H^{1,0}(C) 
& = \ker \rho_{*}/\im (j_{1})_{*}
\cong \tilde{E}_{2}^{0,0}/\im (j_{1})_{*}
= \coker d_{1}.
\qedhere
\end{align*}
\end{proof}

The structure of $H^{n,0}(C)$ for $n=0, 1$ was determined 
in \cite{NR:GeneralizedImageJ}.  
We can also read the following result from the above proof.

\begin{prop}\label{E2term-CESS}
\noindent
For the {\CESS} of \eqref{CESS-for-E_{m+1}^{1}} we have 
\[
\tilde{E}_{2}^{s,0} \cong 
\begin{cases}
\ker \rho_{*} 
       &\mbox{for $s=0$},    \\
\coker \rho_{*} 
\quad
(=H^{2,0}(C))
       &\mbox{for $s=1$},    \\
\CEE{s-1}{j}{B_{m+1}}
       &\mbox{for $s \ge 2$}.
\end{cases}
\]
\end{prop}

Combining this with \eqref{cess-E2-t>0}, 
we have the chart of Cartan-Eilenberg $E_{2}$-terms as in \autoref{eq-CESSpicture}.

\begin{table}[h]
\caption[]{The Cartan-Eilenberg $E_{2}$-term of \eqref{CESS-for-E_{m+1}^{1}}. Here all Ext groups are over $G(m+1)$.}\label{eq-CESSpicture}
\begin{tabular}[]{|c|cccc|} \hline
  &$\vdots $
    &$\vdots $
      &$\vdots $&\\
& & & & \\
$t=2$
  &$\Ext_{}^{0}(\Tmbar{j}\otimes U_{m+1}^{3})$        
    &$\Ext_{}^{1}(\Tmbar{j}\otimes U_{m+1}^{3})$        
      &$\Ext_{}^{2}(\Tmbar{j}\otimes U_{m+1}^{3})$        
        &$\cdots $\\
& & & & \\
$t=1$
  &$\Ext_{}^{0}(\Tmbar{j}\otimes U_{m+1}^{2})$        
    &$\Ext_{}^{1}(\Tmbar{j}\otimes U_{m+1}^{2})$        
      &$\Ext_{}^{2}(\Tmbar{j}\otimes U_{m+1}^{2})$        
        &$\cdots $\\
& & & & \\
$t=0$
  &$\ker \rho_{*}$        
    &$\coker \rho_{*}$
      &$\Ext_{}^{1}(\Tmbar{j}\otimes B_{m+1})$
        &$\cdots $\\
& & & & \\
\hline 
  &$s=0$ 
    &$s=1$ 
      &$s=2$ & \\ \hline
\end{tabular}
\end{table}

\bigskip

Note that the case $s=t=0$ is not interesting here, as we stated before.  
For $\coker \rho_{*}$, we need to recall some results from the other papers.  
For a ${G(m+1)}$-comodule $M$, 
denote the subgroup 
$\bigcap_{n \ge p^{j}} \ker \wrr_{n}$
of $M$ by $L_{j}(M)$. 
Then, 
the map 
\[
(c \otimes 1) \psi: 
L_{j}(M) \longrightarrow \CEE{0}{j}{M}
\]
is an isomorphism between $A(m+1)$-modules
by Lemma 1.12.  
Thus, to obtain the structure of $\tilde{E}_{2}^{1,0}$, 
we may alternatively examine the map 
\[
\rho_{*}: L_{j}(W_{m+1}) \longrightarrow L_{j}(B_{m+1}). 
\]
The following can be read from \cite[Corollary 4.3]{NR:GeneralizedImageJ}.

\begin{lem}\label{lem-structure-of-cokerrho}
The $\coker \rho_{*}$ is isomorphic to the quotient 
\[
L_{j}(B_{m+1})
\left/
\left( A(m+1) \big\{ \bbeta'_{i/i} \mid 0<i\le p^{j-1} \big\} \right)
\right. . 
\]
\end{lem}

The structure of $L_{j}(B_{m+1})$ is defermined in
\cite{NR:OnthePrimitiveBeta} for all $m$ and $j$.  
In particular, the following is the results for $j=2$.

\begin{lem}[{\cite[Theorem 6.1]{NR:OnthePrimitiveBeta}}]
\label{lem-structure-of-B}
Below dimension $p^{3}|\vv_{2}|$, 
$L_{2}(B_{m+1})$ is the $A(m+1)$-module generated by 
\[
\Big\{ \bbeta'_{i/t} \mid i \ge 1, 0<t \le \min (i,p) \Big\} 
\cup 
\Big\{ \bbeta_{ap^{2}+b/t} \mid p<t \le p^{2}, a>0 \mbox{ and } 0 \le b<p \Big\}. 
\]
In particular, below dimension $|\vv_{2}^{p^{2}+1}/v_{1}^{p^2}|$, the
comodule $B_{m+1}$ is $2$-free and $L_{2}(B_{m+1})$ is the
$A(m+1)$-module generated by 

\setcounter{equation}{\value{thm}}
\begin{equation}\label{B-for-j=2}
\Big\{ \bbeta'_{i/\min (i,p)} \mid i>0 \Big\} 
\cup 
\Big\{ \bbeta_{i/t} \mid p<t \le p^2 \le i < p^{2}+p \Big\}.
\end{equation}
\setcounter{thm}{\value{equation}}%
\end{lem}

\section{Extending the range of \texorpdfstring{$E_{m+1}^{2}$}{Em+12}}

In 
Theorem 4.5 
we determined the structure of 
$\Ext_{\Gamma(m+1)}^{*}(BP_{*})$ 
below dimension $p^{2}|\vv_{1}|$.  
Here we extend this range to $p|\vv_{2}|$.  
This is the dimension where the subcomodule $E_{m+1}^{2}$ of $E_{m+1}^{1}/(v_{1}^{\infty})$ 
starts to behave badly for $m>0$.  

\bigskip

By 
Lemma 4.2 
the Poincar\'{e} series of $E_{m+1}^{2}$ below dimension $p|\vv_{2}|$ is
\setcounter{equation}{\value{thm}}
\begin{align}
\lefteqn{g_{m+2} (t)\left(\frac{x^{p} (1-y)}{(1-x^{p}) (1-x_{2})}
+ \frac{x^{p^{2}} (1-y^{p+1})}{(1-x^{p^{2}}) (1-x_{3})}\right)}
\qquad\qquad \nonumber\\
& =
g (BP_{*}/I_{2})\frac{x^{p} }{(1-x^{p}) (1-x_{2})}
+
g (BP_{*}/I_{2})\Frac{x^{p^{2}}(1-y^{p+1})}{1-y}. 
\label{lem-E2series-again}
\end{align}
\setcounter{thm}{\value{equation}}%
The first term corresponds to the module described in 
Theorem 4.5, 
and the second term presumably corresponds to 
\[
BP_{*}/(p,v_{1})
\left\{
\bbeta_{p/j,p+2-j}
\mid 0<j\le p \right\}.
\]

\noindent
We see that 
\[
\bbeta_{p/j,p+2-j}
= \frac{\vv_{2}^{p}}{p^{p+2-j}v_{1}^{j}}
= \sum_{0 \le k<j} 
\binom{p}{k}\frac{p^{j-2-k}}{v_{1}^{j-k}}
\llambda_{2}^{p-k} w^{k}
\quad
\in E_{m+1}^{1}/(v_{1}^{\infty})
\]
(where $w$ is as in \eqref{eq-w}) for $j \ge 2$, 
but $\bbeta_{p/1,p+1} \not\in E_{m+1}^{1}/(v_{1}^{\infty})$.  
We get around this problem by replacing $\bbeta_{p/1,p+1}$ with 
\[
\tilde{\bbeta}_{p/1,p+1}=
\frac{\vv_{2}^{p}}{p^{p+1}v_{1}} - \frac{\vv_{3}}{pv_{1}^{2}}+
\frac{v_{2}\vv_{2}^{p}}{pv_{1}^{p+2}}
- \Frac{v_{2}^{p^{m+1}}\vv_{1}}{p^{2}v_{1}^{2}}
\quad
\in E_{m+1}^{1}/(v_{1}^{\infty}).
\]
Then, our extension of 
Theorem 4.5 for $m>0$ is the following.

\begin{thm}\label{thm-E2extended}
Let $E^{2}_{m+1}$ be the $A(m+2)$-module generated by the set 
\[
\left\{\bbeta_{i/j,k} \mid i+1 \ge j+k 
\right\} 
\cup 
\left\{\bbeta_{p/j,p+2-j} \mid 2 \le j \le p \right\}
\cup 
\left\{\tilde{\bbeta}_{p/1,p+1} \right\}.
\]
Below dimension $p|\vv_{2}|$, it has the {\Ps} specified in \eqref{lem-E2series-again}, 
it is a sub $\Gamma(m+1)$-comodule of $E^{1}_{m+1}/(v_{1}^{\infty})$, 
and its Ext group is isomorphic to 
\[
A(m+1)/I_{2}\otimes 
E(\hh_{1,0}) \otimes P(\bb_{1,0}) \otimes 
\left\{\bbeta'_{i},  \bbeta_{p/k}
\mid i \ge 1, 2\le k \le p
\right\}.
\]
In particular $\Ext^{0}$ maps monomorphically to
$\Ext_{\Gamma(m+1)}^{2}(BP_{*})$ in that range.
\end{thm}

\begin{proof}
Define a decreasing filtration on $BP_{*}/(p^{\infty},v_{1}^{\infty})$
by $\vv_{2}^{a}/p^{b}v_{1}^{c}\in F^{n}$ if and only if $a-b-c \ge n$.
Then, each element of the first set belongs to $F^{-1}$
and the submodule generated by the set is a subcomodule.  
We also see that the reduced expansion of $\bbeta_{p/j,p+2-j}$ is in $F^{-1}$
though $\bbeta_{p/j,p+2-j}$ itself is belonging to $F^{-2}$, 
and the reduced expansion of $\tilde{\bbeta}_{p/1,p+1}$ is in $F^{-2}$.  
Thus the module generated by the assigned set is a comodule
as desired.

The Ext group can be computed similarly to the proof of 
Theorem 4.5.  
\end{proof}

\begin{rem}\label{failure-5.2}
From \eqref{ses-DE-2}, 
we have the {\LES}: 
\[
\xymatrix
@R=8mm
@C=10mm
{
0 \ar[r]^{}
    & \Ext^{0}(E_{m+1}^{1})
             \ar[r]^-{(i_{2})_{*}} 
        & \Ext^{0}(D_{m+1}^{1})
                \ar[r]^-{(j_{2})_{*}}  
             & \Ext^{0}(E_{m+1}^{2})
                   \ar `r[d] `[l] `[llld]_-{\delta^{1}} `[dll] [dll]    \\
    & \Ext^{1}(E_{m+1}^{1})
                   \ar[r]^-{(i_{2})_{*}}  & \cdots, &
}
\]
where all Ext groups are over $\Gamma (m+1)$.  
As we have seen in Lemma 4.1, 
the map $(i_{2})_{*}$ induces an isomorphism in
$\Ext^{0}$ for $m=0$.  
However, for $m>0$, we have a non-trivial element 
$pv_{1}\tilde{\bbeta}_{p/1,p+1} (=-v_{2}^{p^{m+1}}\vv_{1}/pv_{1})$ 
in $\ker\delta^{1}$.  
This is actually the first such element 
and the map $(i_{2})_{*}$ is still isomorphic and
$\Ext_{\Gamma (m+1)}^{0}(E_{m+1}^{2}) $ is isomorphic to $\Ext_{\Gamma
(m+1)}^{2}(BP_{*})$
below this dimension ({\ie} $p|\vv_{2}|$).  
\end{rem}

\section{Quillen operations of some elements}
Recall that the Quillen operation 
$\wrr_{j}: M \to \Sigma^{j |\wt_{1}|} M$
for $G(m+1)$-comodul $M$ is defined by  
\[
\psi(x) = \sum_{j} \wt_{1}^{j} \otimes \wrr_{j}(x) + \cdots .
\]
In the following sections we will need the action of some Quillen operations 
on $M=U_{m+1}^{*}$ 
to compute the Cartan-Eilenberg $E_{2}$-terms 
$\tilde{E}_{2}^{s,t}$ $(t \ge 1)$ of \autoref{eq-CESSpicture}.

\bigskip

A translation of \autoref{thm-E2extended} to the present context is the following.

\begin{cor}\label{cor-E2again}
Below dimension $p|\vv_{3}|$, 
we have an isomorphism 
\[
U_{m+1}^{*+2}
\cong
E(\hh_{2,0}) \otimes P(\bb_{2,0}) \otimes U_{m+1}^{2}
\]
where $U_{m+1}^{2}$ is isomorphic to the $A(m+1)/I_{2}$-module generated by 
\setcounter{equation}{\value{thm}}
\begin{equation}\label{elements-of-U}
\left\{
\uu_{i,j} = \delta^{0}\delta^{1}\left(\Frac{\vv_{2}^{j}\vv_{3}^{i}}{i! \, pv_{1}}\right), 
\uu_{p/k} = \delta^{0}\delta^{1}\left(\Frac{\vv_{3}^{p}}{pv_{1}^{k}}\right)
\mid 
0<i \le p, j \ge 0, 2\le k\le p 
\right\}
\end{equation}
\setcounter{thm}{\value{equation}}%
and $\delta^{0}$ and $\delta^{1}$ are the connecting homomorphisms
for the {\SES}s
\[
0 \to BP_{*} \to M^{0} \to N^{1} \to 0 
\qquad
and 
\qquad 
0 \to N^{1} \to M^{1} \to N^{2} \to 0
\]
respectively.  
The bidegrees of elements are 
$|\hh_{2,0}| = (1,|\wt_{2}|)$
and
$|\bb_{2,0}| = (2,|\wt_{2}^{p}|)$.  
\end{cor}

In particular, we have 
\[
U_{m+1}^{2a+\varepsilon} 
\cong \bb_{2,0}^{a-1} \otimes \hh_{2,0}^{\varepsilon} \otimes U_{m+1}^{2}
\qquad
\mbox{for $a \ge 1$ and $\varepsilon = 0, 1$}. 
\]
So, it is sufficient to know the Quillen operations on $U_{m+1}^{2}$.  
Instead, we here compute the Quillen operation on 
$\Ext_{\Gamma(m+2)}^{0}(E_{m+1}^{1}/(v_{1}^{\infty}))$
after pulling back elements of \eqref{elements-of-U} by the composition of connecting homomorphisms: 
\setcounter{equation}{\value{thm}}
\begin{equation}\label{connecting-composition}
\xymatrix{
\Ext_{\Gamma(m+2)}^{0}(E_{m+1}^{1}/(v_{1}^{\infty})) \ar@{->>}[r]^-{\delta^{1}}
    & \Ext_{\Gamma(m+2)}^{1}(E_{m+1}^{1}) \ar[r]^-{\delta^{0}}_-{\cong} 
        & U_{m+1}^{2} .
}
\end{equation}
\setcounter{thm}{\value{equation}}%
The corresponding elements will be denoted by $\ttheta_{i,j}$ and
$\ttheta_{p/k}$.

\begin{rem}
The choise of $\ttheta_{i,j}$ is not unique: 
the definition of $\ttheta_{i,j}$ has ambiguity up to elements of $\ker\delta^{1}$.  
In particular, 
the comodule $B_{m+1}$ is involved in $\ker\delta^{1}$ 
and we may tack any element of $B_{m+1}$ to $\ttheta_{i,j}$. 
\end{rem}

Recall the recursive formula (3.10) 
for $\wl_{i}$, 
which are independent on $m$: 
\setcounter{equation}{\value{thm}}
\begin{equation}\label{recursive-lambda}
\wl_{1} = \llambda_{1}, 
\quad
\wl_{2} = \llambda_{2} + \ell_{1}\llambda_{1}^{p}, 
\quad
\wl_{3} = \llambda_{3} + \ell_{1}\llambda_{2}^{p} + \ell_{2}\llambda_{1}^{p^2}. 
\end{equation}
\setcounter{thm}{\value{equation}}%
On the other hand, 
the expression of $\vv_{i}$ in terms of $\llambda_{i}$ depends on $m$.  
For small values of $i$, we have

\begin{lem}\label{prop-vv-and-llambda}
In $D_{m+1}^0$ we have  
\begin{align*}
\vv_1 & = 
p \llambda_1, \\
\vv_2 & = 
p \llambda_2
+ (1-p^{p-1})v_1 \llambda_1^p
- v_1^{p^{m+1}}\llambda_1
\quad (m \ge 1),       \\
\vv_3 & \equiv 
p \llambda_3 - p^{p^2-1} v_2 \llambda_1^{p^2} + \zeta
\mod (v_1)
\quad
\mbox{where}
\quad
\zeta = 
v_{2}\llambda_{1}^{p^2} -
\begin{cases}
0    & (m=1), \\
v_{2}^{p^{m+1}}\llambda_{1}
     & (m \ge 2).
\end{cases}
\end{align*}
\end{lem}

\begin{proof}
By 
(3.9) 
we have 
\begin{align*}
p \wl_{1} & = \vv_{1},  \\
\quad
p \wl_{2} & = \vv_{2} + \ell_{1}\vv_{1}^{p} + \wl_{1}\vv_{1}^{p^{m+1}} 
\quad
(m \ge 1),  \\
\quad
p \wl_{3} & = \vv_{3} + \ell_{1}\vv_{2}^{p} + \ell_{2}\vv_{1}^{p^2}
+ 
\begin{cases}
v_{1}^{p^{m+2}} \wl_{2} 
  & (m=1), \\
v_{1}^{p^{m+2}} \wl_{2} + v_{2}^{p^{m+}} \wl_{1} 
  & (m \ge 2).
\end{cases}
\end{align*}
The result follows from \eqref{recursive-lambda} and the relations between $\ell_{i}$ and $v_{i}$.  
\end{proof}

Define the element $\xi$ in $D_{m+1}^{0}$ by 
\begin{align*}
\xi & =
v_{2}\vv_{2}^{p}-
\begin{cases}
0
& (m=1),    \vspace{5pt} \\
v_{1}^{p}v_{2}^{p^{m+1}}\llambda_{1}
& (m \ge 2). 
\end{cases}
\end{align*}

\begin{lem}\label{zeta-xi}
For $m \ge 1$, we have 
$v_{1}^{p}\zeta  \equiv \xi$
mod $(p^{2},v_{1}^{p^{m+1}})$ in $E_{m+1}^{1}$. 
\end{lem}

\begin{proof}
Note that 
$\vv_{2} \equiv v_{1}\llambda_{1}^{p}$
mod $(p,v_{1}^{p^{m+1}})$. 
For $m \ge 2$
\[
v_{1}^{p}\zeta = 
v_{2}(v_{1}\llambda_{1}^{p})^{p} 
- v_{1}^{p}v_{2}^{p^{m+1}}\llambda_{1}
\equiv 
v_{2}\vv_{2}^{p} 
- v_{1}^{p}v_{2}^{p^{m+1}}\llambda_{1}
= \xi 
\]
mod $(p^{2},v_{1}^{p^{m+1}})$.  
The case $m=1$ is similarly proved. 
\end{proof}

\begin{prop}\label{def-up-1}
Define $\ttheta_{p,j}$ for $j \ge 0$ by 
\setcounter{equation}{\value{thm}}
\begin{equation}\label{def-up-1,j}
\ttheta_{p,j}=
\vv_{2}^{j}
\left(
\Frac{\vv_{3}^{p}}{p! \cdot p v_{1}}
- \Frac{\xi^{p}}{p! \cdot p v_{1}^{1+p^2}} 
\right) .
\end{equation}
\setcounter{thm}{\value{equation}}%
Then it is in 
$\Ext_{\Gamma(m+2)}^{0}(E_{m+1}^{1}/(v_{1}^{\infty}))$ 
and satisfies 
$\delta^{0}\delta^{1}(\ttheta_{p,j}) = \uu_{p,j}$. 
\end{prop}

\begin{proof}
By \autoref{zeta-xi} we see that 
\[
\Frac{\vv_{2}^{j}\vv_{3}^{p}}{p! \cdot p v_{1}}
\equiv
\frac{\vv_{2}^{j}(p\llambda_{3} - p^{p^2-1}v_{2}\llambda_{1}^{p^2} + \zeta)^{p}}{p! \cdot p v_{1}}
\equiv
\frac{\vv_{2}^{j}(v_{1}^{p}\zeta)^{p}}{p! \cdot p v_{1}^{1+p^2}}
\equiv
\frac{\vv_{2}^{j}\xi^{p}}{p! \cdot p v_{1}^{1+p^2}}
\]
mod $E_{m+1}^{1}/(v_{1}^{\infty})$.  
Direct calculations show that $\ttheta_{p,j}$ is invariant over $\Gamma(m+2)$.  
Since $v_{1}^{-p^2-1}\xi^{p}/p^2$ is in $\ker\delta^{1}$, 
the second statement follows. 
\end{proof}

\begin{prop}\label{prop-ttheta-and-QuillenOp}
Define $\ttheta_{i,j}$ for $0<i \le p$ and $j \ge 0$ 
by \eqref{def-up-1,j} 
and the downward induction on $i$: 
\[
\ttheta_{i,j} =  v_{2}^{-1}\wrr_{p^2}(\ttheta_{i+1,j}) 
\qquad
\mbox{for $0<i<p$}.
\]
Then they are in 
$\Ext_{\Gamma(m+2)}^{0}(E_{m+1}^{1}/(v_{1}^{\infty}))$
and satisfy 
$\delta^{0}\delta^{1}(\ttheta_{i,j})=\uu_{i,j}$.  
\end{prop}

\begin{proof}
The first statement is obvious 
since $\Ext_{\Gamma(m+2)}^{0}(E_{m+1}^{1}/(v_{1}^{\infty}))$ is a subcomodule of $E_{m+1}^{1}/(v_{1}^{\infty})$. 
Since the second term of \eqref{def-up-1,j}  is in $\ker\delta^{1}$ 
and each Quillen operation commutes with the connecting homomorphism, 
the second statement follows.  
\end{proof}

The following lemma on the Quillen operation is useful.  

\begin{lem}\label{lem-iteration-quillen}
The $k$-fold iteration of $\wrr_{p^{j}}$ is congruent to 
$k! \, \wrr_{kp^{j}}$ modulo $p^{j}$.  
\end{lem}

\begin{proof}
Since 
$r_{s}r_{t} = \binom{s+t}{s} r_{s+t}$, 
the $k$-fold iteration of $\wrr_{p^{j}}$ is equal to 
\[
\frac{(kp^{j})!}{(p^{j}!)^{k}} \wrr_{kp^{j}},
\]
where the coefficient is congruent to $k!$ modulo $p^{j}$.
\end{proof}

Then we have

\begin{prop}\label{thm-uij}
Quillen operations on $\ttheta_{1,j}$ for $0 \le j \le p^2-p$ 
are given by 
\[
\wrr_{p^{2}}(\ttheta_{1,j}) = 0
\quad
\mbox{and}
\quad
\wrr_{p}(\ttheta_{1,j}) = j v_{2} \bbeta_{j+p-1/p}
\]
up to unit scalar multiplication. 
\end{prop}

\begin{proof}
By \autoref{lem-iteration-quillen} 
$\wrr_{p^{2}}(\ttheta_{1,j})$ is a unit multiple of $v_{2}^{-p+1} \wrr_{p^{3}}(\ttheta_{p,j})$, 
and we can check 
$\wrr_{p^{3}}(\ttheta_{p,j}) = 0$.  
Similarly, $\wrr_{p}(\ttheta_{1,j})$ is a unit multiple of $v_{2}^{-p+1} \wrr_{p^{3}-p^{2}+p}(\ttheta_{p,j})$, 
which can be computed by direct calculation.   
\end{proof}

\begin{prop}\label{comodulestructure-u}
We have
\[
\psi(\ttheta_{i,j}) \equiv 
\sum_{0 \le k <i} \wt_{1}^{kp^{2}} \otimes \frac{v_{2}^{k}\ttheta_{i-k,j}}{k!} 
\quad
\mod(v_{2}^{i}) .
\]
\end{prop}

\begin{proof}
Roughly speaking, it follows from 
$k! \, \wrr_{kp^{2}}(\ttheta_{i,j}) = v_{2}^{k}\ttheta_{i-k,j}$
since 
$\wrr_{p^2}(\ttheta_{i+1,j}) = v_{2}\ttheta_{i,j}$. 
More precisely, 
it is enough to consider 
$\psi(v_{2}^{p-i}\ttheta_{i,j})$ mod $(v_{2}^{p})$ 
using the equality 
$v_{2}^{p-i}\ttheta_{i,j} = (p-i)! \, \wrr_{(p-i)p^{2}}(\ttheta_{p,j})$. 
\end{proof}

\begin{prop}\label{lem-u_{p/k}}
Define $\ttheta_{p/k}$ $(0<k\le p)$ by 
\[
\ttheta_{p/k} = 
\Frac{\vv_{3}^{p}}{pv_{1}^{k}}
- \Frac{v_{2}^{p}\vv_{2}^{p^2}}{pv_{1}^{p^{2}+k}}
+ \Frac{v_{2}^{p^{m+2}}\vv_{2}}{pv_{1}^{k+1}} .
\]
Then it is in $\Ext_{\Gamma(m+2)}^{0}(E_{m+1}^{1}/(v_{1}^{\infty}))$ 
and satisfies $\delta^{0}\delta^{1}(\ttheta_{p/k})=\uu_{p/k}$. 
Moreover, it is $G(m+1)$-invariant: 
we have
$\wrr_{j}(\ttheta_{p/k})=0$
for all $j \ge 1$. 
\end{prop}

\begin{proof}
By \autoref{prop-vv-and-llambda}, 
modulo $E_{m+1}^{1}/(v_{1}^{\infty})$
\begin{align*}
\ttheta_{p/k}
& \equiv 
\frac{
v_{1}^{p}\llambda_{2}^{p^2} 
+ v_{2}^{p}\llambda_{1}^{p^3}}{p v_{1}^{k}}
- \frac{v_{2}^{p}\cdot v_{1}^{p^2}\llambda_{1}^{p^3}}{p v_{1}^{p^2+k}}
+ \frac{v_{2}^{p^{m+2}}\cdot v_{1}\llambda_{1}^{p}}{p v_{1}^{k+1}} \\
& \equiv 
\frac{v_{1}^{p}\llambda_{2}^{p^2}}{p v_{1}^{k}}
+ \frac{(p\llambda_{1})^{p^{m+2}}\cdot \llambda_{1}^{p}}{p v_{1}^{k}}
\equiv
0 
\end{align*}
for $m=1$, and 
\begin{align*}
\ttheta_{p/k}
& \equiv 
\frac{
v_{1}^{p}\llambda_{2}^{p^2} 
+ v_{2}^{p}\llambda_{1}^{p^3}
- v_{2}^{p^{m+2}}\llambda_{1}^{p}}{p v_{1}^{k}}
- \frac{v_{2}^{p}\cdot v_{1}^{p^2}\llambda_{1}^{p^3}}{p v_{1}^{p^2+k}}
+ \frac{v_{2}^{p^{m+2}}\cdot v_{1}\llambda_{1}^{p}}{p v_{1}^{k+1}}
\equiv
0 
\end{align*}
for $m \ge 2$.  
The second statement follows since 
all terms in $\ttheta_{p/k}$ except for the leading term are in $\ker\delta^{1}$. 
The last statement follows from direct calculations. 
\end{proof}

\section{The homotopy groups of \texorpdfstring{$T(m)_{(2)}$}{T(m)(2)}}

In this section we determine the homotopy groups of $T(m)_{(2)}$ below dimensions $p|\vv_{3}|$
by analyzing the Cartan-Eilenberg $E_{2}$-term of \autoref{eq-CESSpicture} for $j=2$. 
By \autoref{lem-structure-of-cokerrho} and \ref{lem-structure-of-B} we have

\begin{prop}\label{lem-B2free}
Below dimension $|\vv_{2}^{p^{2}+1}/v_{1}^{p^{2}}|$, 
the Cartan-Eilenberg $E_{2}$-term of \autoref{eq-CESSpicture} for $j=2$ satisfies 
$\tilde{E}_{2}^{s,0} = 0$
for $s \ge 2$, 
and $\tilde{E}_{2}^{1,0}$ is isomorphic to the $A(m+1)$-module generated by 
\[
\Big\{ \bbeta'_{i/t} \mid i \ge 2, 0<t \le \min (i-1,p) \Big\} 
\cup 
\Big\{ \bbeta_{p^{2}/t} \mid p <t \le p^{2} \Big\} .
\]
Note that $|\vv_{2}^{p^{2}+1}/v_{1}^{p^{2}}|$ is larger than $p|\vv_{3}|$ if $m>0$.  
\end{prop}

Thus our remaining task is to determine the structure of 
\[
\tilde{E}_{2}^{s,t} \cong \CEE{s}{2}{U_{m+1}^{t+1}} 
\quad
\mbox{for $t \ge 1$.}
\]
Since this is a certain suspension of $\tilde{E}_{2}^{s,1}$
({\ie} tensored object with some power of $\bb_{2,0}$ and $\hh_{2,0}$), 
we may consider only for $\tilde{E}_{2}^{s,1}$. 
Below dimension $p|\vv_{3}|$, 
define the $v_{2}$-torsion free $A(m+1)$-submodule $U^{0}$ of $v_{2}^{-1}U_{m+1}^{2}$ 
by adjoining the elements
\[
\left\{ v_{2}^{-i}\uu_{i,j} \mid 0<i \le p, j\ge 0 \right\}
\cup 
\left\{ v_{2}^{-p}\uu_{p/k} \mid 2 \le k \le p \right\}
\]
to $U_{m+1}^{2}$.  
Note that $U^{0}$ is a comodule 
since the congruence in \autoref{comodulestructure-u} is modulo $v_{2}^{i}$ 
and the ignored elements have non-negative $v_{2}$-exponent after applying $v_{2}^{-i}$.  
We also define the quotient comodule $U^{1}$ by the following {\SES}: 
\setcounter{equation}{\value{thm}}
\begin{equation}\label{eq-USES}
0 \longrightarrow U_{m+1}^{2} \longrightarrow U^{0} \longrightarrow U^{1} \longrightarrow 0  
\end{equation}
\setcounter{thm}{\value{equation}}%

The Quillen operations on $v_{2}^{-p}\uu_{p/k} \in U^{0}$ are trivial
by \autoref{lem-u_{p/k}}.  
The behavior of Quillen operations on $v_{2}^{-i}\uu_{i,j} \in U^{0}$ follows from
\autoref{prop-ttheta-and-QuillenOp}, 
and it is demonstrated in \eqref{diagram-elements-R0} for $p=5$, 
where each diagonal arrow represents the action of
$\wrr_{p^{2}}$ up to unit scalar multiplication and the elements in
the rightmost column are out of our range except for $j=0$.
\setcounter{equation}{\value{thm}}
\begin{equation}\label{diagram-elements-R0}
\begin{split}
\xymatrix@=10pt{
\vdots 
    & \vdots 
        & \vdots 
            & \vdots 
                & \vdots \\
\uu_{1,j} 
    & \uu_{2,j} 
        & \uu_{3,j} 
            & \uu_{4,j} 
                & \uu_{5,j} \\
v_{2}^{-1}\uu_{1,j} 
    & v_{2}^{-1}\uu_{2,j} \ar[ul] 
        & v_{2}^{-1}\uu_{3,j} \ar[ul] 
            & v_{2}^{-1}\uu_{4,j} \ar[ul] 
                & v_{2}^{-1}\uu_{5,j} \ar[ul] \\
    & v_{2}^{-2}\uu_{2,j} \ar[ul] 
        & v_{2}^{-2}\uu_{3,j} \ar[ul] 
            & v_{2}^{-2}\uu_{4,j} \ar[ul] 
                & v_{2}^{-2}\uu_{5,j} \ar[ul] \\
    &   & v_{2}^{-3}\uu_{3,j} \ar[ul] 
            & v_{2}^{-3}\uu_{4,j} \ar[ul] 
                & v_{2}^{-3}\uu_{5,j} \ar[ul]   \\
    &   &   & v_{2}^{-4}\uu_{4,j} \ar[ul] 
                & v_{2}^{-4}\uu_{5,j} \ar[ul] \\
    &   &   &   & v_{2}^{-5}\uu_{5,j} \ar[ul] 
}
\end{split}
\end{equation}
\setcounter{thm}{\value{equation}}%

\begin{prop}\label{Ext-R0}
$U^{0}$ is $2$-free, 
and we have an isomorphism of $A(m+1)$-modules 
\begin{align*}
\CEE{0}{2}{U^{0}}
& \cong 
A(m+1) \otimes \left\{ v_{2}^{-1}\uu_{1,j}, v_{2}^{-p}\uu_{p/k} \mid j \ge 0, 2 \le k \le p \right\}.
\end{align*}
\end{prop}

\begin{proof}
By 
Lemma 1.12, 
$\CEE{0}{2}{U^{0}}$ is additively isomorphic to 
\[
L_{2}(U^{0}) = \bigcap_{\ell \ge p^2} \ker \wrr_{\ell}.
\] 
In \eqref{diagram-elements-R0} the only possible elements with trivial action of $\wrr_{p^2}$
are $v_{2}^{-1}\uu_{1,j}$.    
Note that 
\[
\wrr_{\ell}(v_{2}^{-1}\uu_{1,j}) = \delta^{0}\delta^{1}(v_{2}^{-1}\wrr_{\ell}(\ttheta_{1,j}))
\]
and 
$v_{2}^{-1}\wrr_{\ell}(\ttheta_{1,j}) = 0$ 
for $\ell \neq 1, p^{2}$ 
because 
\begin{align*}
\psi\left(\Frac{\vv_{2}^{j}\vv_{3}}{pv_{1}}\right)
& = 
\Frac{\vv_{2}^{j}(\vv_{3}+v_{2}\wt_{1}^{p^2}-v_{2}^{p^{m+1}}\wt_{1})}{pv_{1}}.
\end{align*}
Indeed, we have 
$\wrr_{\ell}(v_{2}^{-1}\uu_{1,j}) = 0$
even for $\ell = 1$ or $p^{2}$ because  
\[
v_{2}^{-1}\wrr_{1}(\ttheta_{1,j}) 
= v_{2}^{p^{m+1}-1}\bbeta_{j}
\quad
\mbox{and}
\quad 
v_{2}^{-1}\wrr_{p^2}(\ttheta_{1,j}) = \bbeta_{j}
\]
are in $\ker\delta^{1}$. 
Thus all Quillen operations on $v_{2}^{-1}\uu_{1,j}$ are trivial.  
Note that it is also shown that there is a bijection between 
$\CEE{0}{2}{U^{0}}$ and $\Ext_{G(m+1)}^{0}(U^{0})$.

The diagram \eqref{diagram-elements-R0} also suggests the equality of Poincar\'{e} series 
\begin{align*}
g(U^{0}) & = 
\frac{g(\Ext_{}^{0}(U^{0}))}{1-x^{p^2}} 
\qquad
\mbox{where}
\quad 
x=t^{|\vv_{1}|}
\end{align*}
and we have 
\begin{align*}
g(\Tmbar{2} \otimes U^{0})
& = g(U^{0}) \cdot \Frac{1-x^{p^2}}{1-x}
= \Frac{g(\Ext^{0}(U^{0}))}{1-x}  \\
& = g(\Ext^{0}(U^{0})) \cdot g(G(m+1)/I)   \\
& = g(\Ext^{0}(\Tmbar{2} \otimes U^{0})) \cdot g(G(m+1)/I) 
\end{align*}
which means that $U^{0}$ is 2-free. 
\end{proof}

\begin{prop}\label{Ext-R1}
$U^{1}$ is $2$-free, 
and we have an isomorphism of $A(m+1)$-modules 
\begin{align*}
\CEE{0}{2}{U^{1}}
& \cong
A(m+1)/I_{3} \otimes \left\{ \uu_{i,j}/v_{2} \mid i \ge 1, j \ge 0  \right\}.
\end{align*}
\end{prop}

\begin{proof}
The analogous diagram to \eqref{diagram-elements-R0} for $p=5$ is as follows: 
\[
\xymatrix@=10pt{
\uu_{1,j}/v_{2} & \uu_{2,j}/v_{2} & \uu_{3,j}/v_{2} & \uu_{4,j}/v_{2} & \uu_{5,j}/v_{2}  \\
 & \uu_{2,j}/v_{2}^{2} \ar[ul] & \uu_{3,j}/v_{2}^{2} \ar[ul] & \uu_{4,j}/v_{2}^{2} \ar[ul] & \uu_{5,j}/v_{2}^{2} \ar[ul] & \\
 & & \uu_{3,j}/v_{2}^{3} \ar[ul] & \uu_{4,j}/v_{2}^{3} \ar[ul] & \uu_{5,j}/v_{2}^{3} \ar[ul]  \\
 & & & \uu_{4,j}/v_{2}^{4} \ar[ul] & \uu_{5,j}/v_{2}^{4} \ar[ul] \\
 & & & & \uu_{5,j}/v_{2}^{5} \ar[ul] 
}
\]
In this case $\Ext^{0}$ is generated by the elements in the top row. 
The 2-freeness of $U^{1}$ is similarly shown to $U^{0}$.  
\end{proof}

\begin{prop}\label{lem-U2free}
Below dimension $p|\vv_{3}|$, 
the Cartan-Eilenberg $E_{2}$-term of \autoref{eq-CESSpicture} for $j=2$ satisfies 
\[
\tilde{E}_{2}^{s,*+1}
\cong 
E(\hh_{2,0}) \otimes P(\bb_{2,0}) \otimes \CEE{s}{2}{U_{m+1}^{2}}
\]
and
\begin{align*}
\tilde{E}_{2}^{s,1}
& = \CEE{s}{2}{U_{m+1}^{2}}   \\
& \cong 
\begin{cases}
A(m+1)/I_{2} \otimes \left\{ \uu_{1,i}, \uu_{p/k} \mid i \ge 0, 2 \le k \le p \right\}
     & \mbox{for $s=0$,}   \vspace{5pt}\\
A(m+2)/I_{3} \otimes \left\{ \ggamma_{\ell} \mid \ell \ge 2 \right\}
     & \mbox{for $s=1$,}       \vspace{5pt}\\
0    & \mbox{for $s \ge 2$}
\end{cases}
\end{align*}
where
$\ggamma_{\ell} = \delta^{2}\left( \uu_{\ell,0}/v_{2} \right)$
and $\delta^{2}$ is the connecting homomorphism associated to \eqref{eq-USES}. 
The operators behave as if they had bidegree 
$\hh_{2,0} \in \tilde{E}_{2}^{0,1}$
and
$\bb_{2,0} \in \tilde{E}_{2}^{0,2}$. 
\end{prop}

\begin{proof}
By \autoref{Ext-R0} and \ref{Ext-R1}, we have the 4-term exact sequence\footnote{The case $m=0$ was described in \cite[Lemma 7.3.5]{Rav:MU2nd}.}
\[
0
\longrightarrow
\tilde{E}_{2}^{0,1}
\longrightarrow
\CEE{0}{2}{U^{0}}
\longrightarrow
\CEE{0}{2}{U^{1}}
\longrightarrow
\tilde{E}_{2}^{1,1}
\longrightarrow
0
\]
and 
$\tilde{E}_{2}^{s,1} = 0$ for $s \ge 2$.  
Since the image of the middle map is 
\[
A(m+1)/I_{2} \otimes \{ \uu_{1,i}/v_{2} \mid j \ge 0 \} 
\cong 
A(m+2)/I_{3} \otimes \{ \uu_{1,0}/v_{2} \}
\]
we obtain the result.  
\end{proof}

By \autoref{lem-B2free} and \ref{lem-U2free},
\autoref{eq-CESSpicture} is reduced to the following one:   

\begin{table}[h]
\caption[]{The Cartan-Eilenberg $E_{2}$-term of \eqref{CESS-for-E_{m+1}^{1}} for $j=2$.}
\label{eq-CESSpicture-j=2}
\begin{tabular}[]{|c|cccc|} \hline
  &$\vdots $
    &$\vdots $
      &$\vdots $&  \\
& & & & \\
$t=2$
  & $\Ext_{}^{0}(\Tmbar{2}\otimes U_{m+1}^{3})$        
    & $\Ext_{}^{1}(\Tmbar{2}\otimes U_{m+1}^{3})$        
      & $0$        
        & $\cdots $  \\
& & & & \\
$t=1$
  & $\Ext_{}^{0}(\Tmbar{2}\otimes U_{m+1}^{2})$
    & $\Ext_{}^{1}(\Tmbar{2}\otimes U_{m+1}^{2})$
      & $0$        
        & $\cdots $  \\
& & & & \\
$t=0$
  & $\ker \rho_{*}$        
    & described in \autoref{lem-B2free}
      & $0$
        &$\cdots$  \\
& & & & \\
\hline 
  & $s=0$ 
    & $s=1$ 
      & $s=2$ &  \\ \hline
\end{tabular} 
\end{table}

\begin{prop}\label{ANSSE2-j=2}
Below dimension $p|\vv_{3}|$, 
the {\CESS} of \autoref{eq-CESSpicture} for $j=2$ collapses, 
and we have the {\SES}
\[
0 \longrightarrow 
\tilde{E}_{\infty}^{1,t} \longrightarrow 
\Ext_{\Gamma(m+1)}^{t+2}(\Tm{2}) \longrightarrow 
\tilde{E}_{\infty}^{0,t+1} \longrightarrow 
0
\]
which splits for $t \ge 1$, but for $t=0$.
\end{prop}

\begin{proof}
The {\SS} collapses
since we have only two columns in \autoref{eq-CESSpicture-j=2}.  
The middle groups is isomorphic to 
$\Ext_{\Gamma(m+1)}^{t+1}(\Tm{2}\otimes E_{m+1}^{1})$, 
and the {\SES}s follow by inspection of \autoref{eq-CESSpicture-j=2}.  
For $t \ge 1$, it splits because 
$\tilde{E}_{2}^{1,t}$ is $v_{2}$-torsion 
while $\tilde{E}_{2}^{0,t+1}$ is $v_{2}$-torsion free by \autoref{lem-U2free}.    
For $t=0$, for example, an element 
\[
\uu_{1,0}
\in 
\CEE{0}{2}{U_{m+1}^{2}}
\cong \tilde{E}_{2}^{0,1} 
\]
is killed by $v_{1}$, however, its lift  
\[
\delta^{0}\delta^{1}(\ttheta_{1,0})
= 
\delta^{0}\delta^{1}
\left(
  \frac{\vv_{3}}{pv_{1}}
- \frac{v_{2}\vv_{2}^{p}}{pv_{1}^{1+p}} 
\right)
\in \Ext_{\Gamma(m+1)}^{2}(\Tm{2})
\]
is not killed by $v_{1}$. 
Thus, it does not split.  
\end{proof}

\begin{thm}\label{cor-collapse}
Below dimension $p|\vv_{3}|$, the {\ANSS} for $T(m)_{(2)}$ collapses.
\end{thm}

\begin{proof}
We have computed the Adams-Novikov 
$E_{2}^{n,*} = \Ext_{\Gamma(m+1)}^{n}(\Tm{2})$ for $n \ge 2$ 
and the shortest possible differential is 
$d_{2p-1}: E_{2}^{2,*} \to E_{2}^{2p+1,*}$. 
The first element in the target is
$\hh_{2,0}\bb_{2,0}^{p-1}\uu_{1,0} \in E_{2}^{2p+1,*}$, 
and its total degree 
\[
2(p^{m+4}+p^{m+2}-p^2-p)-3 
\]
is larger than $p|\vv_{3}|$.  
\end{proof}

\section{The homotopy groups of \texorpdfstring{$T(m)_{(1)}$}{T(m)(1)}}

In this section we determine the homotopy groups of $T(m)_{(1)}$ below dimensions $p|\vv_{3}|$. 
To determine the Cartan-Eilenberg $E_{2}$-term of \autoref{eq-CESSpicture} for $j=1$, 
we use the algebraic {\SDSS} of Theorem 1.17: 
For a $G(m+1)$-comodule $M$ and non-negative integer $i$, 
there is a {\SS} converging to 
$\Ext_{G(m+1)}(\Tmbar{i} \otimes_{A(m+1)} M)$
with 
\[
E_{1}^{*,t} \cong
E(\hh_{1,j}) \otimes P(\bb_{1,j}) \otimes \Ext_{G(m+1)}^{t}(\Tmbar{i+1} \otimes_{A(m+1)} M)
\]
with 
$\hh_{1,j} \in E_{1}^{1,0}$, 
$\bb_{1,j} \in E_{1}^{2,0}$, 
and 
$d_{r}: E_{r}^{s,t} \to E_{r}^{s+r,t-r+1}$.   
In particular, $d_{1}$ is induced by the action on $M$ of $r_{p^{j}}$ 
for $s$ even and $r_{(p-1)p^{j}}$ for $s$ odd. 
Note that $r_{(p-1)p^{j}}$ is congruent to the $(p-1)$-fold iteration of $r_{p^j}$
up to unit scalar multiplication.

\bigskip

The case $M = U_{m+1}^{2}$ is easy.

\begin{prop}\label{lem-U2-j=1}
Below dimension $p|\vv_{3}|$, 
the algebraic {\SDSS} for $U_{m+1}^{2}$ collapses
from the $E_{2}$-term, and
\[
\CEE{*+k}{1}{U_{m+1}^{2}} 
\cong 
E(\hh_{1,1}) \otimes P(\bb_{1,1}) \otimes \CEE{k}{2}{U_{m+1}^{2}}. 
\]
\end{prop}

\begin{proof}
Since the action of $\wrr_{p}$ on $U_{m+1}^{2}$
is trivial by \autoref{cor-E2again}, 
the $E_{1}$-term coincides
with the $E_{2}$-term.  The differentials 
$d_{2}: E_{2}^{s,1} \to E_{2}^{s+2,0}$ 
are also trivial since the source is $v_{2}$-torsion
while the target is $v_{2}$-torsion free.  By \autoref{lem-U2free} the
{\SDSS} has only two rows, and so $d_{r} = 0$ for $r \ge 3$.
\end{proof}

Hereafter we will denote $\uu_{1,i}$ by $\uu_{i}$ for short.  
Since 
\[
\tilde{E}_{2}^{s,t} 
\cong \CEE{s}{1}{U_{m+1}^{t+1}} 
\quad
\mbox{for $t \ge 1$}, 
\]
the following is a translation of \autoref{lem-U2-j=1}.  

\begin{cor}\label{cor-ExtU}
Below dimension $p|\vv_{3}|$, 
the Cartan-Eilenberg $E_{2}$-term of \autoref{eq-CESSpicture}
\[
\tilde{E}_{2}^{*+s,*+1} \cong \CEE{*+s}{1}{U_{m+1}^{*+2}} 
\]
is isomorphic to 
\[
E(\hh_{2,0}, \hh_{1,1}) \otimes P(\bb_{2,0}, \bb_{1,1}) \otimes 
\left\{
\begin{array}{c}
A(m+1)/I_{2} \otimes \Big\{ \uu_{i}, \uu_{p/k} \mid i \ge 0, 2 \le k \le p \Big\} \\
\oplus \\
A(m+2)/I_{3} \otimes \Big\{ \ggamma_{\ell} \mid \ell \ge 2 \Big\}
\end{array}
\right.
\] 
where the bidegree of elements are 
$\uu \in \tilde{E}_{2}^{0,1} $
and
$\ggamma \in \tilde{E}_{2}^{1,1}$
and the operators behave as if they had the bidegree 
$\hh_{2,0} \in \tilde{E}_{2}^{0,1}$, 
$\bb_{2,0} \in \tilde{E}_{2}^{0,2}$, 
$\hh_{1,1} \in \tilde{E}_{2}^{1,0}$
and
$\bb_{1,1} \in \tilde{E}_{2}^{2,0}$. 
\end{cor}

\bigskip

The algebraic {\SDSS} for $M=B_{m+1}$ was treated in \cite{NR:OnthePrimitiveBeta}, 
which we summarize here.  
Below dimension $|\vv_{2}^{p^{2}+1}/v_{1}^{p^2}|$ 
it collapses from $E_{2}$-term since $B_{m+1}$ is $2$-free by \autoref{lem-structure-of-B}, 
so we need to compute only $d_{1}$.  
On the elements of $\CEE{0}{2}{B_{m+1}}$ \eqref{B-for-j=2}, we have 
\[
\wrr_{p}(\bbeta'_{i/e_{1}}) = \bbeta_{i-1/e_{1}-1}, 
\quad
\wrr_{p}(\bbeta_{pi/e_{1}}) = 0 
\quad
\mbox{and}
\quad
\wrr_{p^2-p}(\bbeta'_{i/p}) = \bbeta_{i-p+1/1} 
\]
up to unit scalar multiplication 
(cf.~\cite[Proposition B.2]{NR:OnthePrimitiveBeta}). 
It may be helpful to demonstrate the behavior of $d_{1}$ for $p=3$. 
The following diagrams describes $d_{1}$ related to the first set of \eqref{B-for-j=2}: 

\setcounter{equation}{\value{thm}}
\begin{equation}\label{diagram-CESS-d1-1}
\begin{split}
\xymatrix@=15pt{
\bbeta'_{3/1} & \bbeta'_{3/2} \ar[dl]_{\wrr_{3}} & \bbeta'_{3/3} \ar[dl]_{\wrr_{3}}  \ar@/^1pc/[ddll]^{\wrr_{6}}  \\
\bbeta'_{2/1} & \bbeta'_{2/2} \ar[dl]_{\wrr_{3}}  \\
\bbeta'_{1/1}
}
\xymatrix@=15pt{
 & & \bbeta_{5/3} \ar[dl]^{\wrr_{3}} \ar@/_1pc/[ddll]_{\wrr_{6}}  \\
 & \bbeta_{4/2} \ar[dl]^{\wrr_{3}} & \bbeta_{4/3} \ar[dl]^{\wrr_{3}}  \\
\bbeta_{3/1} & \bbeta_{3/2} & \bbeta_{3/3}
}
\end{split}
\end{equation}
\setcounter{thm}{\value{equation}}%
Corresponding to the diagonal containing $\bbeta'_{1/1}$, the subgroup of $E_{1}$ generated by 
\[
E(\hh_{1,1}) \otimes P(\bb_{1,1}) \otimes 
\{ \bbeta'_{1/1}, \ldots, \bbeta'_{p/p} \} 
\]
reduces to simply $\{ \bbeta'_{1/1} \}$ on passage to $E_{2}$.  
The similar argument is true for the diagonal containing $\bbeta_{p/1}$.  
On the other hand, 
corresponding to the diagonal containing $\bbeta'_{i/1}$ $(2 \le i \le p)$ 
is the subgroup generated by 
\[
E(\hh_{1,1}) \otimes P(\bb_{1,1}) \otimes 
\{ \bbeta'_{i/1}, \ldots, \bbeta'_{p/p-i+1} \}
\]
which is reduced to 
$P(\bb_{1,1}) \otimes \{ \bbeta'_{i/1}, \hh_{1,1}\bbeta'_{p/p-i+1} \}$.  
The similar argument is true for the diagonal containing 
$\bbeta_{p/i}$ $(2 \le i \le p)$; 
the subgroup generated by 
\[
E(\hh_{1,1}) \otimes P(\bb_{1,1}) \otimes 
\{ \bbeta_{p/i}, \ldots, \bbeta_{2p-i/p} \}
\]
reduces to 
$P(\bb_{1,1}) \otimes \{ \bbeta_{p/i}, \hh_{1,1}\bbeta_{2p-i/p} \}$.  
In particular, the subgroups corresponding to $\bbeta'_{p/1}$ and $\bbeta_{p/p}$ survive to $E_{2}$ entirely.

\begin{rem}\label{quillen-massey-rel}
In the diagram \eqref{diagram-CESS-d1-1} we can read off the existence of certain Massey products.  
For example, if we have a relation
$\wrr_{p}(b)=a$, then it we have the Massey product
$\langle \hh_{1,1}, \hh_{1,1}, a \rangle$, 
as we will explain in \autoref{app-massey}.  
In general, if we have a sequence 
\setcounter{equation}{\value{thm}}
\begin{equation}\label{sequence-for-Massey}
a_{i} \stackrel{\wrr_{p}}{\longrightarrow}
a_{i-1} \stackrel{\wrr_{p}}{\longrightarrow}
\cdots \stackrel{\wrr_{p}}{\longrightarrow}
a_{1}
\qquad
(0<i<p)
\end{equation}
\setcounter{thm}{\value{equation}}%
then we would have the Massey product 
$\langle \hh_{1,1}, \ldots, \hh_{1,1}, a_{1} \rangle$ 
with $i$-factors of $\hh_{1,1}$
whose representative has the leading term $\wt_{1}^{p} \otimes a_{i}$.
In this paper we denote this Massey product by $\specialmassey{i}{a_{1}}$, although it is denoted by $\underline{pi}a_{1}$ in \cite[Definition 7.4.12]{Rav:MU2nd}.  
\end{rem}

Note that the entire configuration is $\vv_{2}^{p}$-periodic. 
The diagram containing $\bbeta_{p^2/1}$ corresponding to the right one of \eqref{diagram-CESS-d1-1} is combined with the diagram for the second set of \eqref{B-for-j=2}: 
\setcounter{equation}{\value{thm}}
\begin{equation}\label{diagram-CESS-d1-2}
\begin{split}
\xymatrix@=15pt{
 & & \bbeta_{11/3} \ar[dl]^{\wrr_{3}} \ar@/_1.2pc/[ddll]_{\wrr_{6}} & \cdots & \cdots & \bbeta_{11/9} \ar[dl]^{\wrr_{3}} \ar@/_1.2pc/[ddll]_{\wrr_{6}}  \\
& \bbeta_{10/2} \ar[dl]^{\wrr_{3}} & \cdots & \cdots & \bbeta_{10/8} \ar[dl]^{\wrr_{3}} & \bbeta_{10/9} \ar[dl]^{\wrr_{3}}  \\
\bbeta_{9/1} & \cdots & \cdots & \bbeta_{9/7} & \bbeta_{9/8} & \bbeta_{9/9}  \\
}
\end{split}
\end{equation}
\setcounter{thm}{\value{equation}}%
Then, the summand corresponding to $\bbeta_{p^2/k}$ $(1 \le k \le p^2-p+1)$ reduces to $\{ \bbeta_{p^2/k} \}$, 
and the summand corresponding to $\bbeta_{p^2/p^2-\ell}$ $(0 \le \ell \le p-2)$ reduces to 
$P(\bb_{1,1}) \otimes \{ \bbeta_{p^2/p^2-\ell}, \hh_{1,1} \bbeta_{p^2+\ell/p^2} \}$.

\bigskip

By these observations we have the following result:

\begin{prop}[{\cite[Proposition 7.3]{NR:OnthePrimitiveBeta}}]\label{prop-ExtB}
Below dimensions $|\vv_{2}^{p^2+1}/v_{1}^{p^2}|$,
the Cartan-Eilenberg $E_{2}$-term of \autoref{eq-CESSpicture}
\[
\tilde{E}_{2}^{*+1,0} = \CEE{*}{1}{B_{m+1}}
\]
has the following $A(m+1)/I_{2}$-basis:
\[
\begin{array}{c}
P(\vv_{2}^{p}) \otimes \big\{ \bbeta'_{1},\bbeta_{p/1} \big\}
\oplus 
\big\{ \bbeta_{p^{2}/k} \mid 1 \le k \le p^2-p+1 \big\}  \\
\oplus  \\
P(\bb_{1,1}) \otimes 
\left(
\begin{array}{c}
P(\vv_{2}^{p}) \otimes
\big\{ 
\bbeta'_{i/1}, 
\hh_{1,1}\bbeta '_{p/p-i+1}, 
\bbeta_{p/i}, 
\hh_{1,1}\bbeta_{2p-i/p} \mid 2 \le i \le p
\big\}  \\
\oplus  \\
\big\{ 
\bbeta_{p^{2}/p^{2}-\ell},  
\hh_{1,1} \bbeta_{p^2+\ell/p^2} 
\mid 0 \le \ell \le p-2
\big\} 
\end{array}
\right) 
\end{array}
\]
subject to the caveat that $\vv_{2}\bbeta_{k/e}=\bbeta_{k+1/e}$.  The
bigrading of elements are (omitting unnecessary subscripts) $\bbeta
\in \tilde{E}_{2}^{1,0}$ and the operators $\hh_{1,1}$ and $\bb_{1,1}$
behave as if they had the bidegrees given in \autoref{cor-ExtU}.  
\end{prop}

Note that the range of dimensions 
({\ie} $|\vv_{2}^{p^2+1}/v_{1}^{p^2}|$) 
exceeds $p|\vv_{3}|$ for $m>0$.

\bigskip

Now we have determined the Cartan-Eilenberg $E_{2}$-term for $j=1$.
In the followings we will see that the {\SS} has a rich pattern of differentials, 
which is essentially independent of $m$.

For the differential 
\[
\tilde{d}_{2}: 
\tilde{E}_{2}^{s,1} = 
\CEE{s}{1}{U_{m+1}^{2}} \longrightarrow 
\tilde{E}_{2}^{s+2,0} = 
\CEE{s+1}{1}{B_{m+1}}
\]
we may ignore the $v_{2}$-torsion part of the source
({\ie}$\gamma$-elements) since the target is $v_{2}$-torsion free.  
For the other part, we have the following result\footnote{The result for $m=0$ was described in \cite[Lemma 7.3.12]{Rav:MU2nd}.}.

\begin{lem}\label{d2-CESS-j=1}
The {\CESS} of \autoref{eq-CESSpicture} for $j=1$ has the following differentials: 
\begin{enumerate}
\def\theenumi{\roman{enumi}}
\item\label{d2-CESS-j=1-(i)}
$\tilde{d}_{2}(\uu_{i})
= i v_{2}\hh_{1,1}\bbeta_{i+p-1/p}
$
for $i \nequiv 0$ mod $p$.  

\item\label{d2-CESS-j=1-(ii)}
$\tilde{d}_{2} (\hh_{1,1}\uu_{i})
= \displaystyle\binom{i}{p-1} v_{2}\bb_{1,1}\bbeta_{i+1/2}$
for $i \equiv -1$ mod $p$.  

\end{enumerate}

\noindent 
All differentials commute with multiplication by $\bb_{1,1}$.  
\end{lem}

\begin{proof}
We are considering the {\CESS} for $\Tm{1}\otimes_{}E_{m+1}^{1}$, 
and its $\Ext_{}^{s}$ for
$s>0$ is a quotient of (isomorphic to for $s>1$) $\Ext_{}^{s-1}$  for
$\Tm{1}\otimes_{}E_{m+1}^{1}/(v_{1}^{\infty})$, so we can work in
the cobar complex over $G(m+1)$ for the latter comodule.

The differential \eqref{d2-CESS-j=1-(i)} follows from 
$\wrr_{p} (\uu_{i}) = i v_{2}\bbeta_{i+p-1/p}$
given by \autoref{thm-uij}.  
We also have
$\wrr_{p^{2}-p}(\uu_{i}) = \binom{i}{p-1} v_{2}\bbeta_{i+1/2}$ 
and the differential \eqref{d2-CESS-j=1-(ii)}
by \autoref{lem-iteration-quillen}.
\end{proof}

Now the diagram \eqref{diagram-CESS-d1-1}) for $p=3$ is reviewed as follows.  
In each case the graph now has $2p+1$ instead of
$2p$ components, three of which are maximal: 
\setcounter{equation}{\value{thm}}
\begin{equation}\label{eq-graphm>0}
\begin{split}
\xymatrix@=15pt{
\bbeta'_{3/1} 
    & \bbeta'_{3/2} \ar[dl]_{\wrr_{3}} 
        & \bbeta'_{3/3} \ar[dl]_{\wrr_{3}}  \ar@/^1pc/[ddll]^{\wrr_{6}}  \\
\bbeta'_{2/1} 
    & \bbeta'_{2/2} \ar[dl]_{\wrr_{3}}  \\
\bbeta'_{1/1}
}
\xymatrix@=15pt{
    &   & \bbeta_{5/3} \ar[dl]_{\wrr_{3}} \ar@/_1.5pc/[ddll]_{\wrr_{6}} 
            & v_{2}^{-1}\uu_{2} \ar[dl]^{\tilde{d}_{2}} 
                                \ar@/_1pc/[ddll]_{\tilde{d}_{2}}  \\
    & \bbeta_{4/2} \ar[dl]_{\wrr_{3}} 
        & \bbeta_{4/3} \ar[dl]^{\wrr_{3}} 
            & v_{2}^{-1}\uu_{1} \ar[dl]^{\tilde{d}_{2}}  \\
\bbeta_{3/1} 
    & \bbeta_{3/2} 
        & \bbeta_{3/3} 
            & v_{2}^{-1}\uu_{0}
}
\end{split}
\end{equation}
\setcounter{thm}{\value{equation}}%
In fact, each $d_{1}$ in the {\SDSS} behaves as it were the Cartan-Eilenberg $\tilde{d}_{2}$.  
Note that the bigrading of elements in the {\SDSS} are 
$\bbeta \in E_{r}^{0,2}$, 
$\uu \in E_{r}^{0,2}$
and 
$\ggamma \in E_{r}^{0,3}$, 
and each operator has the same bigrading as that for {\CESS}.  
In general, 
the small descent $d_{r}$ correspond to the Cartan-Eilenberg $\tilde{d}_{r+1}$ for $r\geq 1$. 
See \autoref{Table-CESS-j=1}.

\begin{table}[h]
\renewcommand{\arraystretch}{1.5}
\caption[]{Bigradings of elements. 
Some subscripts have been omitted.}\label{Table-CESS-j=1}
\begin{tabular}{|c|c:c:c:c|}
\multicolumn{5}{c}{{\CESS} for $j=1$}  \\ \hline
$t=3$
  & $\bb_{2,0}\uu$
    & $\hh_{1,1}\bb_{2,0}\uu$
  & $\bb_{1,1}\bb_{2,0}\uu$
    & $\hh_{1,1}\bb_{1,1}\bb_{2,0}\uu$  \\ \cdashline{2-5}
$t=2$
  & $\hh_{2,0}\uu$
    & $\hh_{1,1}\hh_{2,0}\uu$
  & $\bb_{1,1}\hh_{2,0}\uu$
    & $\hh_{1,1}\bb_{1,1}\hh_{2,0}\uu$ \\ \cdashline{2-5}
$t=1$
  & $\uu$
    & $\hh_{1,1}\uu$
  & $\bb_{1,1}\uu$
    & $\hh_{1,1}\bb_{1,1}\uu$  \\ \cdashline{2-5}
$t=0$
  & $*$
    & $\bbeta$
  & $\hh_{1,1}\bbeta$
    & $\bb_{1,1}\bbeta$ \\ \hline
  & \multicolumn{1}{c}{$s=0$}
    & \multicolumn{1}{c}{$s=1$}
  & \multicolumn{1}{c}{$s=2$}
    & $s=3$  \\ \hline 
\multicolumn{5}{c}{}    \\
\multicolumn{5}{c}{{\SDSS} for $j=1$}  \\ \hline
$t=4$ 
  &$\bb_{2,0}\uu$ 
    & $\hh_{1,1}\bb_{2,0}\uu$
     & $\bb_{1,1}\bb_{2,0}\uu$
       & $\hh_{1,1}\bb_{1,1}\bb_{2,0}\uu$ \\ \cdashline{2-5}
$t=3$
  &$\hh_{2,0}\uu$ 
    & $\hh_{1,1}\hh_{2,0}\uu$
      & $\bb_{1,1}\hh_{2,0}\uu$
        & $\hh_{1,1}\bb_{1,1}\hh_{2,0}\uu$  \\ \cdashline{2-5}
$t=2$ 
  & $\uu$ 
    & $\hh_{1,1}\uu$
      & $\bb_{1,1}\uu$
        & $\hh_{1,1}\bb_{1,1}\uu$  \\
  & $\bbeta $ 
    & $\hh_{1,1}\bbeta$
      & $\bb_{1,1}\bbeta$
        & $\hh_{1,1}\bb_{1,1}\bbeta$  \\ \cdashline{2-5}
$t=1$ 
  & $*$ 
    & & &  \\ \hline
  & \multicolumn{1}{c}{$s=0$}
    & \multicolumn{1}{c}{$s=1$}
  & \multicolumn{1}{c}{$s=2$}
    & $s=3$  \\ \hline 
\end{tabular}
\end{table}

\begin{rem}\label{rem-renaming}
In \eqref{eq-graphm>0} the ``virtual'' element 
$v_{2}^{-1}\uu_{i}$ 
lives in $\CEE{0}{1}{U^{0}}$
but not in $\CEE{0}{1}{U_{m+1}^{2}}$. 
This means that 
$\hh_{1,1}\bb_{1,1}^{k}\bbeta_{i+p-1/p}$ 
is not actually trivial but $v_{2}$-torsion, 
and that it is chromatically renamed 
$\vv_{2}^{i}\bb_{1,1}^{k}\ggamma_{1}$.  
This is a feature of the cases $m \ge 0$
and it does not happen for $m=0$.  
For example, in the {\CSS} we have 
\begin{align*}
d_{e} (v_{2}^{-1}\uu_{1})
& = 
d_{e}
\left(
\Frac{v_{2}^{-1}\vv_{2}\vv_{3}}{pv_{1}}
- \Frac{\vv_{2}^{p+1}}{pv_{1}^{p+1}} 
\right)
= 
\Frac{\vv_{2}\vv_{3}}{pv_{1}v_{2}} 
= \vv_{2}\ggamma_{1}  \\
\mbox{and}
\qquad
d_{i} (v_{2}^{-1}\uu_{1}) 
& = 
- \Frac{\vv_{2}^{p}\wt_{1}^{p}}{pv_{1}^{p}}
- \Frac{v_{2}^{p^{m+1}-1}\vv_{2}\wt_{1}}{pv_{1}} 
\equiv 
-\hh_{1,1}\bbeta_{p/p}.
\end{align*}
The second term in $d_{i}$ is the product of $\wt_{1}$ with an invariant element $x$.  
It is ignored because we are working in $T(m)_{(1)}$; 
it is the coboundary of $\wt_{1}\otimes_{}x$.

It is also observed that 
$\bb_{1,1}^{k+1}\bbeta_{ip/2}$ 
is renamed 
$\vv_{2}^{ip-1}\hh_{1,1}\bb_{1,1}^{k}\ggamma_{1}$.  
For example, we have 
\begin{align*}
d_{e}(v_{2}^{-1}\hh_{1,1}\uu_{p-1}) 
& = 
d_{e}
\left(
\wt_{1}^{p}
\left(
\Frac{v_{2}^{-1}\vv_{2}^{p-1}\vv_{3}}{pv_{1}}
- \Frac{\vv_{2}^{2p-1}}{pv_{1}^{p+1}} 
\right)
\right)
= \Frac{\vv_{2}^{p-1}\vv_{3}\wt_{1}^{p}}{pv_{1}v_{2}} 
= \vv_{2}^{p-1}\hh_{1,1}\ggamma_{1}  \\
\mbox{and}
\quad
d_{i}(v_{2}^{-1}\hh_{1,1}\uu_{p-1}) 
& = - \wt_{1}^{p} \otimes \Frac{\vv_{2}^{p}\wt_{1}^{p^{2}-p}}{pv_{1}^{2}} + \cdots
= \bb_{1,1}\bbeta_{p/2}.
\end{align*}
\end{rem}

The following result concerns higher Cartan-Eilenberg differentials, 
and we will prove it in the next section.

\begin{thm}\label{thm-7.3.15}
The {\CESS} of \autoref{eq-CESSpicture} for $j=1$ 
has the following differentials 
and no others in our range of dimensions:
\begin{enumerate}
\def\theenumi{\roman{enumi}}
\item\label{thm-7.3.15-1}
$\tilde{d}_{3}(\hh_{2,0}\uu_{i}) = v_{2}\bb_{1,1}\bbeta'_{i+1}$
for $i \nequiv 0$ mod $p$.  

\item\label{thm-7.3.15-2}
$\tilde{d}_{3} (\hh_{2,0}^{\varepsilon }\bb_{2,0}^{k}\uu_{i})
= v_{2}\hh_{1,1}\bb_{1,1}\hh_{2,0}^{\epsilon}\bb_{2,0}^{k-1}\uu_{i-1}$
for $i \nequiv 0$ mod $p$, 
$k \ge 1$ and $\varepsilon =0$ or $1$.  

\item\label{thm-7.3.15-3}
$\tilde{d}_{2k+3}(\hh_{1,1}\hh_{2,0}\bb_{2,0}^{k}\uu_{i}) = v_{2}^{k+1}\hh_{1,1}\bb_{1,1}^{k+1}\bbeta'_{i+1/k+1}$
for $i \equiv -1$ mod $p$ and $0\leq k<p-1$.

\item\label{thm-7.3.15-4}
$\tilde{d}_{2k+2}(\hh_{1,1}\bb_{2,0}^{k}\uu_{i}) = v_{2}^{k+1}\bb_{1,1}^{k+1}\bbeta_{i+1/k+2}$
for $i \equiv -1$ mod $p$ and $1\leq k<p-1$
(the case $k=0$ is \autoref{d2-CESS-j=1}(ii)).

\item\label{thm-7.3.15-5}
$\tilde{d}_{2p-1}(\hh_{1,1}\bb_{2,0}^{p-1}\uu_{i}) = v_{2}^{p-1}\bb_{1,1}^{p}\uu_{i-p+1}$
for $i \equiv -1$ mod $p$.
\end{enumerate}
All differentials commute with multiplication by $\bb_{1,1}$.  
\end{thm}

We will prove \autoref{thm-7.3.15} in the next section.  

\bigskip

Since each source of the stated differentials lies in $\tilde{E}_{r}^{0,*}$ or $\tilde{E}_{r}^{1,*}$, 
it cannot be the target of another differential.
Moreover, each differential has maximal length for the bidegree of its source. 
Thus, the source should be a permanent cycle if a differential is trivial.

\begin{rem}\label{rem-filtration}
We can define a decreasing filtration on $B_{m+1}$ and $U_{m+1}$ by
\[
||\bbeta'_{i/j}|| = i-j-1,
\quad 
||\uu_{i}|| = i+[i/p],
\quad
\mbox{and}
\quad
||p||=||v_{1}||=||v_{2}||=1.
\]
Then the source and target of each differential listed in \autoref{thm-7.3.15} have the same filtration.  
A similar filtration for $m=0$ is discussed in \cite[Lemma 7.4.6]{Rav:MU2nd}.  
In \eqref{eq-graphm>0} all elements along the same diagonal
(e.g., $\bbeta_{2}$, $\bbeta '_{3/2}$, $\bbeta_{3/3}$ and $v_{2}^{-1}\uu_{1}$ in filtration $0$)
have the same filtration. 
\end{rem}

\begin{rem}\label{rem-renaming-2}
Again, we obtained the differentials of the form 
$d_{r}(x) = v_{2}^{t}y$, 
each of which doesn't kill $y$ but makes $y$ into a $v_{2}^{t}$-torsion element, 
as we have already seen in \autoref{rem-renaming}.
For example, the differential in \eqref{thm-7.3.15-1} means that 
$\bb_{1,1}\bbeta'_{i+1}$ is killed by $v_{2}$; 
in the chromatic cobar complex we have
\[
d (v_{2}^{-1}\hh_{2,0}\uu_{i}) 
= - \bb_{1,1}\bbeta'_{i+1} \pm \vv_{2}^{i}\hh_{2,0}\ggamma_{1},
\]
so $\pm \vv_{2}^{i}\hh_{2,0}\ggamma_{1}$ is the new name for $\bb_{1,1}\bbeta'_{i+1}$. 
Similarly,  
$\hh_{1,1}\bb_{1,1}\hh_{2,0}^{\varepsilon}\bb_{2,0}^{k-1}\uu_{i-1}$
is renamed 
$\vv_{2}^{i}\hh_{2,0}^{\varepsilon}\bb_{2,0}^{k}\ggamma_{1}$ 
by \eqref{thm-7.3.15-2}, 
and 
$\hh_{1,1}\bb_{1,1}\bbeta'_{p}$ 
is renamed 
$\vv_{2}^{p-1}\hh_{1,1}\hh_{2,0}\ggamma_{1}$
by \eqref{thm-7.3.15-3}.  
\end{rem}

There are some patterns of differentails 
associated with each component of \eqref{eq-graphm>0}, 
which we now demonstrate for $p=3$.  
For example, for $\bbeta_{2}$ we have the following diagram: 
\[
\xymatrix@=15pt{
\fbox{$\bbeta_{2}$} & \hh_{1,1}\bbeta_{2} & \underline{\bb_{1,1}\bbeta_{2}} & \hh_{1,1}\bb_{1,1}\bbeta_{2} & \underline{\bb_{1,1}^{2}\bbeta_{2}} & \hh_{1,1}\bb_{1,1}^{2}\bbeta_{2}  \\
\bbeta'_{3/2} \ar[ur]^{\wrr_{p}} & 
\fbox{$\hh_{1,1}\bbeta'_{3/2}$} & \bb_{1,1}\bbeta'_{3/2} \ar[ur]^{\wrr_{p}} & \fbox{$\hh_{1,1}\bb_{1,1}\bbeta'_{3/2}$} & \bb_{1,1}^{2}\bbeta'_{3/2} \ar[ur]^{\wrr_{p}} & \underline{\hh_{1,1}\bb_{1,1}^{2}\bbeta'_{3/2}}   \\
 & x_{1} \ar[uur]_{\tilde{d}_{3}} & \fbox{$\hh_{1,1}x_{1}$} & \bb_{1,1}x_{1} \ar[uur]_{\tilde{d}_{3}} & \underline{\hh_{1,1}\bb_{1,1}x_{1}}  \\
 & & & x_{2} \ar[ur]^{\tilde{d}_{3}} & \hh_{1,1}x_{2} \ar[uur]_{\tilde{d}_{5}}
}
\]
where $x_{k}=v_{2}^{-k}\hh_{2,0}\bb_{2,0}^{k-1}\uu_{k}$, and
the boxed elements are permanent in the {\CESS}.  
The underlined elements indeed survive, however, each of these changes into $v_{2}$-torsion element 
(cf.~\autoref{rem-renaming} and \ref{rem-renaming-2}).  
It is also observed that 
$\hh_{1,1}\bbeta'_{3/2}$, $\bb_{1,1}\bbeta_{2}$ and $\hh_{1,1}\bb_{1,1}\bbeta'_{3/2} $
correspond to the Massey products 
$\specialmassey{2}{\bbeta_{2}}$, 
$\specialmassey{1}{\specialmassey{2}{\bbeta_{2}}}$
and 
$\specialmassey{2}{\specialmassey{1}{\specialmassey{2}{\bbeta_{2}}}}$
respectively 
(see \autoref{quillen-massey-rel}).

Similarly, for $\bbeta_{3/3}$ we have the following diagram: 
\[
\xymatrix@=15pt{
\fbox{$\bbeta_{3/3}$} & \underline{\hh_{1,1}\bbeta_{3/3}} & \fbox{$\bb_{1,1}\bbeta_{3/3}$} & \underline{\hh_{1,1}\bb_{1,1}\bbeta_{3/3}} & \underline{\bb_{1,1}^{2}\bbeta_{3/3}}  \\
y_{1} \ar[ur]^{\wrr_{p}} & \fbox{$\hh_{1,1}y_{1}$} & \bb_{1,1}y_{1} \ar[ur]^{\wrr_{p}} & \underline{\hh_{1,1}\bb_{1,1}y_{1}}   \\
 & & y_{2} \ar[ur]^{\tilde{d}_{3}} & \hh_{1,1}y_{2} \ar[uur]_{\tilde{d}_{4}}
}
\]
where $y_{k}=v_{2}^{-k}\bb_{2,0}^{k-1}\uu_{k}$
and $\hh_{1,1}\bbeta_{3/3}$ is renamed $\vv_{2}\ggamma_{1}$, 
and for $\bbeta_{3/2}$ we also have the following diagram: 
\[
\xymatrix@=15pt{
\fbox{$\bbeta_{3/2}$} & \hh_{1,1}\bbeta_{3/2} & \underline{\bb_{1,1}\bbeta_{3/2}}  \\
\bbeta_{4/3} \ar[ur]^{\wrr_{3}} &\underline{ \hh_{1,1}\bbeta_{4/3}}  \\
z \ar[ur]^{d_{2}} & \hh_{1,1}z \ar[uur]_{d_{2}}
}
\]
where 
$z=v_{2}^{-1}\uu_{2}$, 
and we have 
$\hh_{1,1}\bbeta_{4/3} = \specialmassey{2}{\bbeta_{3/2}}$.

\bigskip

Finally, we have the following result:  

\begin{thm}\label{CESS-Einfty}
Below dimension $p|\vv_{3}|$, 
the Cartan-Eilenberg $\tilde{E}_{\infty}$-term of \autoref{eq-CESSpicture} for $j=1$
is the direct sum of the followings: 

\begin{enumerate}
\def\theenumi{\roman{enumi}}
\item\label{CESS-Einfty-1}
the $A(m+1)/I_{2} \otimes P(\vv_{2}^{p})$-module
generated by 
\[
\begin{array}{c}
\left\{ 
\bbeta'_{1}, \bbeta'_{2}, \ldots, \bbeta'_{p};
\bbeta_{p/1}, \bbeta_{p/2};
\hh_{1,1}\bbeta'_{p}
\right\}  \\
\oplus  \\
P(\bb_{1,1}) \otimes 
\left\{ 
\hh_{1,1}\bbeta'_{p/p-i+1}, 
\bbeta_{p/j}
\mid 2 \le i \le p-1, 3 \le j \le p
\right\}   \\
\oplus  \\
E(\hh_{2,0}) \otimes P(\bb_{2,0}) \otimes 
\left(
\begin{array}{c}
P(\bb_{1,1}) \otimes 
\{ \uu_{0} \}  \\
\oplus  \\
\Big\{ \hh_{1,1}\uu_{i} \mid 0 \le i \le p-2 \Big\} 
\end{array}
\right);
\end{array} 
\]

\item\label{CESS-Einfty-2}
the $A(m+1)/I_{3} \otimes P(\vv_{2}^{p})$-module generated by 
\[
\begin{array}{c}
E(\hh_{2,0}) \otimes P(\bb_{1,1},\bb_{2,0}) \otimes 
\left.
\left(
\begin{array}{c}
E(\hh_{1,1}) \otimes 
\{ \vv_{2}^{p-1}\ggamma_{1} \} \\
\oplus  \\
\{ \vv_{2}^{i}\ggamma_{1} \mid 2 \le i \le p-2 \} \\
\oplus   \\
\{ \vv_{2}\ggamma_{1} \}  \\
\end{array}
\right)
\right/ 
\left(
\begin{array}{c}
\vv_{2}^{p-1}\hh_{1,1}\bb_{2,0}\ggamma_{1}, \\
\vv_{2}^{p-1}\bb_{1,1}\bb_{2,0}\ggamma_{1}, \\
\vv_{2}^{i}\bb_{1,1}\bb_{2,0}\ggamma_{1}
\end{array}
\right) 
\end{array}
\]
where the second summand is only for $p \ge 5$;

\item\label{CESS-Einfty-3}
the $A(m+1)/I_{2}$-module generated by 
\[
\begin{array}{c}
\left\{ \bbeta_{p^{2}/k} \mid 1 \le k \le p^2-p+1 \right\}  \\
\oplus  \\
P(\bb_{1,1}) \otimes 
\left(
\begin{array}{c}
\left\{ 
\bbeta_{p^{2}/p^{2}-\ell}, 
\hh_{1,1} \bbeta_{p^2+\ell/p^2} 
\mid 0 \le \ell \le p-2
\right\}  \\
\oplus  \\
E(\hh_{1,1}, \hh_{2,0}) \otimes P(\bb_{2,0}) \otimes 
\Big\{ \uu_{p/k} \mid 2 \le k \le p \Big\} 
\end{array}
\right); 
\mbox{and}
\end{array}
\]

\item\label{CESS-Einfty-4}
the $A(m+2)/I_{3}$-module generated by 
\[
E(\hh_{1,1}, \hh_{2,0}) \otimes P(\bb_{1,1}, \bb_{2,0}) \otimes 
\Big\{ \ggamma_{\ell} \mid \ell \ge 2 \Big\}.
\]
\end{enumerate}

\end{thm}

\begin{rem}
\autoref{thm-7.3.15} \eqref{thm-7.3.15-3} and \eqref{thm-7.3.15-4} 
mean that some elements in the second summand of \autoref{CESS-Einfty} \eqref{CESS-Einfty-1} have higher $v_{2}$-torsion.  
They should be renamed chromatically so as to be realized explicitly that they are $v_{2}$-torsion.  
\end{rem}

Now we have computed 
$\Ext_{\Gamma(m+1)}^{n}(\Tm{1})$
for $n \ge 2$.  
There is no Adams-Novikov differential in this range 
because the first element in filtration $\ge 2p+1$ is 
$\vv_{2}\bb_{1,1}^{p-1}\ggamma_{1}$,
which is not killed by $d_{2p-1}$.
Thus, the {\ANSS} for $T(m)_{(1)}$ collapses 
and \autoref{CESS-Einfty} gives us the stable homotopy groups of $T(m)_{(1)}$.  

\bigskip

The elements for $(p,m)= (3,1)$ are listed in \autoref{fig-2}
and depicted in \autoref{fig-1}.

\begin{figure}[htbp]
\begin{displaymath}
\begin{array}{ccc}
\mbox{\begin{tabular}{|r|c|} \hline 
$t-s$ & Element\\ 
\hline 
\hline 
46  & $\bbeta_{1}$ \\
\hline 
98  & $\bbeta_{2}$ \\
\hline 
142  & $\bbeta_{3/3}$ \\
\hline 
146  & $\bbeta_{3/2}$ \\
\hline 
150  & $\bbeta_{3}$ \\
 & $\bbeta'_{3}$ \\
\hline 
154  & $\uu_{0}$ \\
\hline 
189  & $\vv_{2}\ggamma_{1}$ \\ \hline 
193  & $\hh_{1,1}\bbeta'_{3/2}$ \\
\hline 
197  & $\hh_{1,1}\bbeta'_{3}$ \\
\hline 
201  & $\hh_{1,1}\uu_{0}$ \\
\hline 
202  & $\bbeta_{4}$ \\
\hline 
205  & $\hh_{2,0}\uu_{0}$ \\
\hline 
240  & $\vv_{2}\hh_{2,0}\ggamma_{1}$ \\
\hline 
241  & $\vv_{2}^{2}\ggamma_{1}$ \\
\hline 
252  & $\hh_{1,1}\hh_{2,0}\uu_{0}$ \\
\hline 
253  & $\hh_{1,1}\uu_{1}$ \\
\hline 
254  & $\bbeta_{5}$ \\
\hline 
284  & $\bb_{1,1}\bbeta_{3/3}$ \\
\hline 
288  & $\vv_{2}^{2}\hh_{1,1}\ggamma_{1}$ \\
\hline
292  & $\vv_{2}^{2}\hh_{2,0}\ggamma_{1}$ \\
\hline 
\end{tabular}}
&\quad
\mbox{\begin{tabular}{|r|c|}
\hline 
$t-s$ & Element\\ 
\hline 
\hline 
296  & $\bb_{1,1}\uu_{0}$ \\
\hline 
297  & $\ggamma_{2}$ \\
\hline 
298  & $\bbeta_{6/3}$ \\
\hline 
302  & $\bbeta_{6/2}$ \\
\hline 
304  & $\hh_{1,1}\hh_{2,0}\uu_{1}$ \\
\hline 
306  & $\bbeta'_{6}$ \\
 & $\bbeta_{6}$ \\
\hline 
308  & $\bb_{2,0}\uu_{0}$ \\
\hline 
310  & $\uu_{3}$ \\
\hline 
331  & $\vv_{2}\bb_{1,1}\ggamma_{1}$ \\
\hline 
335  & $\hh_{1,1}\bb_{1,1}\bbeta'_{3/2}$ \\
\hline 
339  & $\vv_{2}^{2}\hh_{1,1}\hh_{2,0}\ggamma_{1}$ \\
\hline 
343  & $\vv_{2}\bb_{2,0}\ggamma_{1}$ \\
\hline 
344  & $\hh_{1,1}\ggamma_{2}$ \\
\hline 
345  & $\vv_{2}^{4}\ggamma_{1}$ \\
\hline 
347  & $\bb_{1,1}\hh_{2,0}\uu_{0}$ \\
\hline 
348  & $\hh_{2,0}\ggamma_{2}$ \\
\hline 
349  & $\hh_{1,1}\bbeta'_{6/2}$ \\
 & $\vv_{2}\ggamma_{2}$ \\
\hline
353  & $\hh_{1,1}\bbeta'_{6}$ \\
\hline 
\end{tabular}}
&\quad
\mbox{\begin{tabular}{|r|c|}
\hline 
$t-s$ & Element\\ 
\hline 
\hline 
355  & $\hh_{1,1}\bb_{2,0}\uu_{0}$ \\
\hline 
357  & $\hh_{1,1}\uu_{3}$ \\
\hline 
358  & $\bbeta_{7}$ \\
\hline 
359  & $\hh_{2,0}\bb_{2,0}\uu_{0}$ \\
\hline 
361  & $\hh_{2,0}\uu_{3}$ \\
\hline 
382  & $\vv_{2}\bb_{1,1}\hh_{2,0}\ggamma_{1}$ \\
\hline 
383  & $\vv_{2}^{2}\bb_{1,1}\ggamma_{1}$ \\
\hline 
394  & $\vv_{2}\hh_{2,0}\bb_{2,0}\ggamma_{1}$ \\
\hline 
395  & $\vv_{2}^{2}\bb_{2,0}\ggamma_{1}$ \\
 & $\hh_{1,1}\hh_{2,0}\ggamma_{2}$ \\
\hline 
396  & $\vv_{2}^{4}\hh_{2,0}\ggamma_{1}$ \\
 & $\vv_{2}\hh_{1,1}\ggamma_{2}$ \\
\hline 
397  & $\vv_{2}^{5}\ggamma_{1}$ \\
\hline 
400  & $\vv_{2}\hh_{2,0}\ggamma_{2}$ \\
\hline 
401  & $\vv_{2}^{2}\ggamma_{2}$ \\
\hline 
406  & $\hh_{1,1}\hh_{2,0}\bb_{2,0}\uu_{0}$ \\
\hline 
407  & $\hh_{1,1}\bb_{2,0}\uu_{1}$ \\
\hline 
408  & $\hh_{1,1}\hh_{2,0}\uu_{3}$ \\
\hline 
409  & $\hh_{1,1}\uu_{4}$ \\
\hline 
410  & $\bbeta_{8}$ \\
\hline\end{tabular}}
\end{array}
\end{displaymath}

\caption{
The elements of $\Ext_{BP_{*}(BP)}^{s,t}(BP_{*} (T(1)_{(1)}))$ for $p=3$, 
and $t-s\le 426$.  
}\label{fig-2}
\end{figure}

\begin{landscape}

\newcommand{\bx}{\put(-1,-1){\line(1,0){2}}
\put(-1,1){\line(1,0){2}}
\put(-1,-1){\line(0,1){2}}
\put(1,-1){\line(0,1){2}}
}

\newcommand{\timesp}{\line(0,4){4}}
\newcommand{\timesv}{\line(4,0){4}}
\newcommand{\masseyh}{{\color{red}\line(96,50){96}}}
\newcommand{\timesh}{{\color{red}\line(48,50){48}}}
\newcommand{\timeshspecial}{{\color{red}\line(48,50){48}\put(48,50){\line(0,4){4}}}}
\newcommand{\masseyhh}{{\color{blue}\line(104,50){104}}}
\newcommand{\timeshh}{{\color{blue}\line(52,50){52}}}
\newcommand{\timeshhspecial}{{\color{blue}\line(52,50){52}\put(52,50){\line(0,4){4}}}}

\unitlength=1.2pt

\begin{figure}[htbp]

\begin{picture}(450,260)(0,-20)
\put(0,0){\line(1,0){450}}
\put(0,280){\line(1,0){450}}
\put(0,0){\line(0,1){280}}
\put(450,0){\line(0,1){280}}
\put(-30,100){\vector(0,1){100}}
\put(100,-30){\vector(1,0){100}}
\put(-33,90){$s$}
\put(90,-33){$t$}
\put(0,50){\line(1,0){5}}
\put(0,100){\line(1,0){5}}
\put(0,150){\line(1,0){5}}
\put(0,200){\line(1,0){5}}
\put(0,250){\line(1,0){5}}
\put(100,0){\line(0,1){5}}
\put(200,0){\line(0,1){5}}
\put(300,0){\line(0,1){5}}
\put(400,0){\line(0,1){5}}
\put(93,-15){100}
\put(193,-15){200}
\put(293,-15){300}
\put(393,-15){400}
\put(-10,47){2}
\put(-10,97){3}
\put(-10,147){4}
\put(-10,197){5}
\put(-10,247){6}
\put(48,50){\circle*{2}}
\put(45,35){$\bbeta_{1}$}
\put(100,50){\circle*{2}\masseyh}
\put(95,35){$\bbeta_{2}$}
\put(144,50){\circle*{2}\timesh\timesv}
\put(125,45){$\bbeta_{3/3}$}
\put(288,150){\circle*{2}\timesh}
\put(148,50){\circle*{2}\masseyh\timesv}
\put(196,100){\circle*{2}\timesh}
\put(340,200){\circle*{2}\timesh}
\put(152,50){\circle*{2}}
\put(152,46){\circle*{2}\timesp\timeshspecial}
\put(200,100){\circle*{2}\masseyh}
\put(204,50){\circle*{2}}
\put(201,35){$\bbeta_{4}$}
\put(256,50){\circle*{2}\masseyh}
\put(251,35){$\bbeta_{5}$}
\put(300,50){\circle*{2}\timesh\timesv}
\put(281,45){$\bbeta_{6/3}$}
\put(304,50){\circle*{2}\masseyh\timesv}
\put(352,100){\circle*{2}\timesh}
\put(308,50){\circle*{2}}
\put(308,46){\circle*{2}\timesp\timeshspecial}
\put(356,100){\circle*{2}}
\put(360,50){\circle*{2}}
\put(357,35){$\bbeta_{7}$}
\put(412,50){\circle*{2}}
\put(407,35){$\bbeta_{8}$}
\put(156,46){\circle*{2}\timeshspecial\timeshhspecial}
\put(156,36){$\uu_{0}$}
\put(204,100){\circle*{2}\timeshh\masseyh}
\put(208,100){\circle*{2}\timesh\masseyhh}
\put(256,150){\circle*{2}\masseyh\masseyhh}
\put(300,150){\circle*{2}\timesh\timeshh}
\put(312,150){\circle*{2}\timesh\timeshh}
\put(352,200){\circle*{2}\timesh}
\put(360,200){\circle*{2}\timeshh}
\put(364,200){\circle*{2}\timesh}
\put(412,250){\circle*{2}}
\put(256,100){\circle*{2}\timeshh}
\put(255,85){$\hh_{1,1}\uu_{1}$}
\put(308,150){\circle*{2}\masseyhh}
\put(412,200){\circle*{2}}
\put(312,46){\circle*{2}\timeshspecial\timeshhspecial}
\put(312,36){$\uu_{3}$}
\put(360,100){\circle*{2}\timeshh}
\put(364,100){\circle*{2}\timesh}
\put(412,150){\circle*{2}}
\put(412,100){\circle*{2}}
\put(192,100){\bx\timeshh\masseyh}
\put(175,105){$\vv_{2}\ggamma_{1}$}
\put(336,200){\bx\timeshh}
\put(244,150){\bx\masseyh\masseyhh}
\put(348,200){\bx\timeshh}
\put(388,250){\bx}
\put(244,100){\bx\timesh\timeshh}
\put(292,150){\bx\timeshh\masseyh}
\put(296,150){\bx\timesh\masseyhh}
\put(388,200){\bx}
\put(400,200){\bx}
\put(344,200){\bx}
\put(400,250){\bx}
\put(348,100){\bx\timeshh}
\put(400,150){\bx}
\put(400,100){\bx}
\put(300,110){\bx\timesh\timeshh}
\put(288,115){$\ggamma_{2}$}
\put(348,160){\bx\timeshh}
\put(352,160){\bx\timesh}
\put(400,210){\bx}
\put(352,110){\bx\timesh\timeshh}
\put(332,115){$\vv_{2}\ggamma_{2}$}
\put(400,160){\bx}
\put(404,160){\bx}
\put(404,110){\bx}
\put(386,115){$\vv_{2}^{2}\ggamma_{2}$}
\put(30,235){
Here we frequently use the relation
}
\put(50,220){
$\bb_{1,1} x 
= \hh_{1,1} \langle \hh_{1,1}, \ldots, \hh_{1,1}, x \rangle$
}
\put(70,205){
$= \langle \hh_{1,1}, \ldots, \hh_{1,1}, \hh_{1,1} x \rangle$
}
\put(30,190){
and the similar one related to $\hh_{2,0}$ and $\bb_{2,0}$.  
}
\end{picture}

\caption[$\Ext_{\Gamma(2)}(T_{1}^{p-1})$]
{
$\Ext_{BP_{*}(BP)}(BP_{*}(T(1)_{(1)}))$
for $p=3$ in dimensions up to 426 dimension.  
\begin{itemize}
\item
Solid dots indicate $v_{2}$-torsion free elements, and squares indicate elements killed by $v_{2}$.
\item
Short vertical and horizontal lines indicate multiplication by $p$ and $v_{1}$.
\item
Red lines (resp.~blue lines) indicate multiplication by $h_{2,1}$ (resp.~$h_{3,0}$) 
and the Massey product operation \\
$\langle h_{2,1},h_{2,1}, - \rangle$ 
(resp.~ $\langle h_{3,0},h_{3,0},- \rangle$). 
The compositition of the two is the multiplication by $b_{2,1}$ (resp.~$b_{3,0}$).
\end{itemize}
}
\label{fig-1}
\end{figure}

\end{landscape}

\section{The proof of \autoref{thm-7.3.15}}

In this section we give a detailed proof\footnote{The case $m=0$ was treated in \cite[\S 7.4]{Rav:MU2nd}.} of \autoref{thm-7.3.15} for $m>0$.
As is stated in the proof of \autoref{d2-CESS-j=1}, our {\SS} is a quotient of the {\CESS} 
and it is enough to prove each differential by computing in $C_{\Gamma(m+1)}(\Tm{1}\otimes_{}N^{2})$.

\begin{lem}\label{lem-b20}
For $m>0$, we have a cocycle
$\bb'_{2,0} = p^{-1}( v_{1}^{p}\bb_{1,1} + d(\wt_{2}^{p}) )$
in the cobar complex over $\Gamma(m+1)$, 
which projects to $\bb_{2,0}$ in that over $\Gamma (m+2)$. 
\end{lem}

\begin{proof}
Recall that we are using the symbols 
$\bb_{1,j}$ and $\bb_{2,0}$ 
for their cobar representatives, namely
\begin{align*}
\bb_{1,j} & = 
p^{-1}d \left(\wt_{1}^{p^{j+1}} \right)
=
- \sum_{0< \ell <p^{j+1}}
p^{-1}
\binom{p^{j+1}}{\ell }
\wt_{1}^{\ell}\otimes_{}\wt_{1}^{p^{j+1}-\ell }  \\
\mbox{and}
\qquad
\bb_{2,0} & \equiv 
p^{-1}\left(\wt_{2}^{p}\otimes_{}1+1\otimes_{}\wt_{2}^{p}
- (\wt_{2}\otimes_{}1+1\otimes_{}\wt_{2})^{p} \right)  \\
& \equiv  
- \sum_{0<\ell <p}p^{-1}\binom{p}{\ell }
\wt_{2}^{\ell }\otimes_{}\wt_{2}^{p-\ell } \qquad \bmod \left(\wt_{1} \right).
\end{align*}
Then the result follows from 
$d(\wt_{2}^{p}) = 
( \wt_{2}^{p} \otimes 1
+ 1 \otimes \wt_{2}^{p}
- (\wt_{2} \otimes 1 
+ v_{1}\bb_{1,0} 
+ 1\otimes_{}\wt_{2})^{p} )$. 
\end{proof}

By Lemma 1.4 and \autoref{lem-b20}, 
it follows that the product of any permanent cycle with $\bb_{2,0}$ is again a permanent cycle.  
This implies that each element in 
\[
\begin{array}{c}
A (m+1)/I_{2} \otimes_{}
E(\hh_{1,1}, \hh_{2,0}) \otimes P(\bb_{1,1}, \bb_{2,0}) \otimes 
\big\{ \uu_{p/k} \mid 2 \le k \le p \big\}  \\
\oplus  \\
A (m+2)/I_{3} \otimes_{}
E (\hh_{1,1}, \hh_{2,0}) \otimes_{}
P(\bb_{1,1}, \bb_{2,0}) \otimes_{} 
\left\{\ggamma_{2},\,\ggamma_{3},\dotsc  \right\}
\end{array}
\]
is a permanent cycle, unlike the case $m=0$.

\begin{lem}\label{eq-dt3}
Let $\overline{\wt}_{3}$ be the conjugation of $\wt_{3}$.  Then we
have
\[
\Delta (\overline{ \wt_{3}}) 
= \overline{ \wt_{3}}\otimes 1 + 1 \otimes \overline{ \wt_{3}}
-v_{1}\bb_{2,0}-v_{2}\bb_{1,1} 
+
\begin{cases}
\wt_{1}^{p^{2}}\otimes_{}\wt_{1}
        &\mbox{for }m=1  \\
0 &\mbox{for }m \ge 2.
\end{cases}
\]
The difference between $\overline{\wt}_{3}$ and $-\wt_{3}$ has trivial
image in $\Gamma (m+2)$.
\end{lem}

\begin{proof}
By definition, 
$\overline{\wt}_{3} = - \wt_{3}+\wt_{1}^{\, 1+p^{2}}$
for $m=1$ 
and 
$\overline{\wt}_{3} = - \wt_{3}$
for $m \ge 2$.  
Since 
\[
\Delta (\wt_{3}) = \wt_{3}\otimes_{}1+1\otimes_{}\wt_{3}
+ v_{1}\bb_{2,0}+v_{2}\bb_{1,1}
+
\begin{cases}
\wt_{1}\otimes_{}\wt_{1}^{p^{2}} 
  &\mbox{for }m=1  \\
0 &\mbox{for }m \ge 2.
\end{cases}
\]
we have the result.  
\end{proof}

\begin{proof}[Proof of \autoref{thm-7.3.15} \eqref{thm-7.3.15-1}]
We may use $\Frac{\vv_{2}^{i}\vv_{3}}{pv_{1}}$ instead of $\uu_{i}$
because these have the same $\delta^{1}\delta^{0}$-image \eqref{connecting-composition} into $U_{m+1}^{2}$.  
For $i>0$, we have
\begin{align*}
d\left(\wt_{2} \otimes_{} 1 \otimes_{}
\frac{\vv_{2}^{i}\vv_{3}}{pv_{1}} \right)
& = 
\wt_{2}\otimes_{} (v_{2}\wt_{1}^{p^{2}}-v_{2}^{p^{m+1}}\wt_{1})
\otimes_{} 1 \otimes_{} 
\Frac{\vv_{2}^{i}}{pv_{1}}, \\
d\left(\wt_{2}\otimes_{} \wt_{1} \otimes_{}
\Frac{v_{2}^{p^{m+1}} \vv_{2}^{i}}{pv_{1}} \right)
& = 
\wt_{2}\otimes_{}\wt_{1} \otimes_{} 1 \otimes_{} 
\Frac{v_{2}^{p^{m+1}} \vv_{2}^{i}}{pv_{1}},  \\
d\left(
\wt_{2}\wt_{1}^{p^{2}} \otimes_{} 1 \otimes_{} \Frac{v_{2}\vv_{2}^{i}}{pv_{1}} 
\right)
& = 
- \left(\wt_{2}\otimes_{}\wt_{1}^{p^{2}}
+ \wt_{1}^{p^{2}}\otimes_{} \wt_{2} \right) \otimes_{} 1 \otimes_{} 
\Frac{v_{2} \vv_{2}^{i}}{pv_{1}},  \\
d\left(\wt_{1}^{p^{2}}\otimes_{} 1 \otimes_{}
\Frac{v_{2}\vv_{2}^{i+1}}{(i+1) p^{2}v_{1}}  \right)
& =
\wt_{1}^{p^{2}}\otimes_{} \wt_{2}\otimes_{} 1 \otimes_{}
\Frac{v_{2}\vv_{2}^{i}}{pv_{1}}   
+\bb_{1,1} \otimes_{} 1\otimes_{} \Frac{v_{2}\vv_{2}^{i+1}}{(i+1) pv_{1}}. 
\end{align*}

\noindent
The sum of the preimages on the left represents $\hh_{2,0}\uu_{i}$; 
summing on the right gives the result. 
\end{proof}

\begin{proof}[Proof of \autoref{thm-7.3.15} \eqref{thm-7.3.15-2}]
We give the proof for $k=1$ and $\varepsilon=1$.  
The general case follows by replacing $\bb_{2,0}$ by $\bb'_{2,0}$ 
(\autoref{lem-b20}) 
and tensoring all equations on the left with the cocycle $(\bb'_{2,0})^{k-1}$.

We have 
$\eta_{R} (\vv_{2}) \equiv \vv_{2}+z$ mod $I^{p^{m+1}}$, 
where 
$I= (p,v_{1}, \ldots)$ and $z=v_{1}\wt_{1}^{p}+p\wt_{2}$.
By this and \autoref{eq-dt3} we have
\begin{align*}
d (\bb_{2,0} \otimes_{} 1\otimes_{}\uu_{i} )
& = \bb_{2,0} \otimes_{} d\left(1\otimes_{}\uu_{i} \right)\\
& = - \bb_{2,0} \otimes_{} v_{2}\sum_{0<k<p}
\binom{i+p}{k}z^{k} \otimes_{} 1 \otimes_{} \Frac{\vv_{2}^{i+p-k}}{\binom{i+p}{p} pv_{1}^{p+1}}    \\
& = - \bb_{2,0} \otimes_{} v_{2}
\left( (i+p) \wt_{1}^{p} \otimes_{} 1 \otimes_{} \Frac{\vv_{2}^{i+p-1}}{\binom{i+p}{p} pv_{1}^{p}}
+ \dotsb
\right), \\
\lefteqn{
d\left(-\wt_{3} \otimes v_{2}
\left(
-(i+p) \wt_{1}^{p} \otimes_{} 1 \otimes_{} \Frac{\vv_{2}^{i+p-1}}{\binom{i+p}{p} pv_{1}^{p+1}}
+ \dotsb
\right)
\right)
} \hspace{68pt}\\
& = 
- (v_{1}\bb_{2,0}+v_{2}\bb_{1,1}) \otimes v_{2} 
\left(-(i+p)\wt_{1}^{p}
\otimes_{}1\otimes_{} \Frac{\vv_{2}^{i+p-1}}{\binom{i+p}{p} pv_{1}^{p+1}}+\dotsb    
\right)\\
& \qquad -\wt_{3}\otimes -v_{2}(i+p)\wt_{1}^{p}
\otimes_{} \binom{i+p-1}{p} \wt_{1}^{p^{2}}
\otimes_{}1\otimes_{} \Frac{\vv_{2}^{i-1}}{\binom{i+p}{p} pv_{1}} \\
& = 
-\bb_{2,0} \otimes_{} 
v_{2}\left(-(i+p)\wt_{1}^{p}
\otimes_{}1\otimes_{} \Frac{\vv_{2}^{i+p-1}}{\binom{i+p}{p} pv_{1}^{p}}+\dotsb\right)\\
& 
\qquad 
-v_{2}\bb_{1,1} \otimes_{} v_{2}
\left(-(i+p)\wt_{1}^{p}
\otimes_{}1\otimes_{} \Frac{\vv_{2}^{i+p-1}}{\binom{i+p}{p} pv_{1}^{p+1}}+\dotsb    
\right)\\
& 
\qquad 
+iv_{2}\wt_{3}\otimes_{} \wt_{1}^{p}\otimes_{} \wt_{1}^{p^{2}}
\otimes_{}1\otimes_{} \Frac{\vv_{2}^{i-1}}{pv_{1}},
\end{align*}
and 
\begin{align*}
\lefteqn{d\left(-i\wt_{3}\otimes_{}\wt_{1}^{p}\otimes_{}1\otimes_{}
\Frac{\vv_{2}^{i-1}\vv_{3}}{pv_{1}}    \right)}\qquad\qquad\\ 
& = 
-iv_{2}\bb_{1,1}\otimes_{}\wt_{1}^{p}\otimes_{}1\otimes_{}
\Frac{\vv_{2}^{i-1}\vv_{3}}{pv_{1}}
-i\wt_{3}\otimes_{}\wt_{1}^{p}\otimes_{}v_{2}\wt_{1}^{p^{2}}\otimes_{}
1\otimes_{}\Frac{\vv_{2}^{i-1}}{pv_{1}}. 
\end{align*}
The sum of the preimages on the left represents $\bb_{2,0}\uu_{i}$, 
and the terms on the right add up to
\begin{align*}
\lefteqn{\bb_{1,1}\otimes_{}\wt_{1}^{p}\otimes_{}1\otimes_{}
\left(
-\Frac{iv_{2}\vv_{2}^{i-1}\vv_{3}}{pv_{1}}
 +\Frac{(i+p)v_{2}^{2}\vv_{2}^{i+p-1}}{\binom{i+p}{p} pv_{1}^{p+1}}
\right) 
 +\dotsb    
}\qquad\qquad \\
& = 
-iv_{2}\bb_{1,1}\otimes_{}\wt_{1}^{p}\otimes_{}1\otimes_{}
\left(\Frac{\vv_{2}^{i-1}\vv_{3}}{pv_{1}}
- \Frac{\vv_{2}^{i+p-1}}{\binom{i+p-1}{p} pv_{1}^{p+1}}
\right)
+\dotsb 
\end{align*}

\noindent
The inspection of $\tilde{E}_{2}$-terms described in \autoref{cor-ExtU} 
shows that the element represents 
$-iv_{2}\hh_{1,1}\bb_{1,1}\uu_{i-1}$ as claimed.  
\end{proof}

To derive \eqref{thm-7.3.15-3}, \eqref{thm-7.3.15-4} and \eqref{thm-7.3.15-5} from \eqref{thm-7.3.15-1} and \eqref{thm-7.3.15-2}, 
we use Massey product arguments.  
Oberve \autoref{fig-moreexamples} for $p=5$, 
in which each diagonal is similar to \eqref{sequence-for-Massey}
and the arrows labeled $\tilde{d}_{r}$ are related to Cartan-Eilenberg differentials given in \autoref{d2-CESS-j=1} and \eqref{thm-7.3.15-2}; 
for example, the differential 
$\tilde{d}_{3}(\bb_{2,0}\uu_{4}) = v_{2}\hh_{1,1}\bb_{1,1}\uu_{3}$
is denoted 
\[
\bb_{2,0}\uu_{4} \stackrel{\tilde{d}_{3}}{\longrightarrow} v_{2}\bb_{1,1}\uu_{3}.
\]

\begin{figure}
\[
\xymatrix@=10pt{
\bbeta_{5} & \bbeta_{6} & \bbeta_{7} & \bbeta_{8} & \bbeta_{9} & \bbeta'_{10}  \\
v_{2}\bbeta_{5/2} & \bbeta_{6/2} \ar[ul]_{\wrr_{3}} & \bbeta_{7/2} \ar[ul]_{\wrr_{3}} & \bbeta_{8/2} \ar[ul]_{\wrr_{3}} & \bbeta_{9/2} \ar[ul]_{\wrr_{3}} & \bbeta'_{10/2} \ar[ul]_{\wrr_{3}} \\
v_{2}^{2}\bb_{1,1}\bbeta_{5/3} & v_{2}\bbeta_{6/3} \ar[ul]_{\wrr_{3}} & \bbeta_{7/3} \ar[ul]_{\wrr_{3}} & \bbeta_{8/3} \ar[ul]_{\wrr_{3}} & \bbeta_{9/3} \ar[ul]_{\wrr_{3}} & \bbeta'_{10/3} \ar[ul]_{\wrr_{3}} \\
v_{2}^{3}\bb_{1,1}^{2}\bbeta_{5/4} & v_{2}^{2}\bb_{1,1}\bbeta_{6/4} \ar[ul]_{\wrr_{3}} & v_{2}\bbeta_{7/4} \ar[ul]_{\wrr_{3}} & \bbeta_{8/4} \ar[ul]_{\wrr_{3}} & \bbeta_{9/4} \ar[ul]_{\wrr_{3}} & \bbeta'_{10/4} \ar[ul]_{\wrr_{3}}  \\
v_{2}^{4}\bb_{1,1}^{3}\bbeta_{5/5} & v_{2}^{3}\bb_{1,1}^{2}\bbeta_{6/5} \ar[ul]_{\wrr_{3}} & v_{2}^{2}\bb_{1,1}\bbeta_{7/5} \ar[ul]_{\wrr_{3}} & v_{2}\bbeta_{8/5} \ar[ul]_{\wrr_{3}} & \bbeta_{9/5} \ar[ul]_{\wrr_{3}} & \bbeta'_{10/5} \ar[ul]_{\wrr_{3}} \\
v_{2}^{4}\bb_{1,1}^{4}\uu_{0} & v_{2}^{3}\bb_{1,1}^{3}\uu_{1} \ar[ul]_{\tilde{d}_{2}} & v_{2}^{2}\bb_{1,1}^{2}\uu_{2} \ar[ul]_{\tilde{d}_{2}} & v_{2}\bb_{1,1}\uu_{3} \ar[ul]_{\tilde{d}_{2}} & \uu_{4} \ar[ul]_{\tilde{d}_{2}} \\
 & v_{2}^{3}\bb_{1,1}^{3}\bb_{2,0}\uu_{1} \ar[ul]_{\tilde{d}_{3}} & v_{2}^{2}\bb_{1,1}^{2}\bb_{2,0}\uu_{2} \ar[ul]_{\tilde{d}_{3}} & v_{2}\bb_{1,1}\bb_{2,0}\uu_{3} \ar[ul]_{\tilde{d}_{3}} & \bb_{2,0}\uu_{4} \ar[ul]_{\tilde{d}_{3}} \\
 & & v_{2}^{2}\bb_{1,1}^{2}\bb_{2,0}^{2}\uu_{2} \ar[ul]_{\tilde{d}_{3}} & v_{2}\bb_{1,1}\bb_{2,0}^{2}\uu_{3} \ar[ul]_{\tilde{d}_{3}} & \bb_{2,0}^{2}\uu_{4} \ar[ul]_{\tilde{d}_{3}} \\
& & & v_{2}\bb_{1,1}\bb_{2,0}^{3}\uu_{3} \ar[ul]_{\tilde{d}_{3}} & \bb_{2,0}^{3}\uu_{4} \ar[ul]_{\tilde{d}_{3}} \\
& & & & \bb_{2,0}^{4}\uu_{4} \ar[ul]_{\tilde{d}_{3}} 
}
\]
\caption[]{Diffetentials for the case $p=5$.}
\label{fig-moreexamples}
\end{figure}
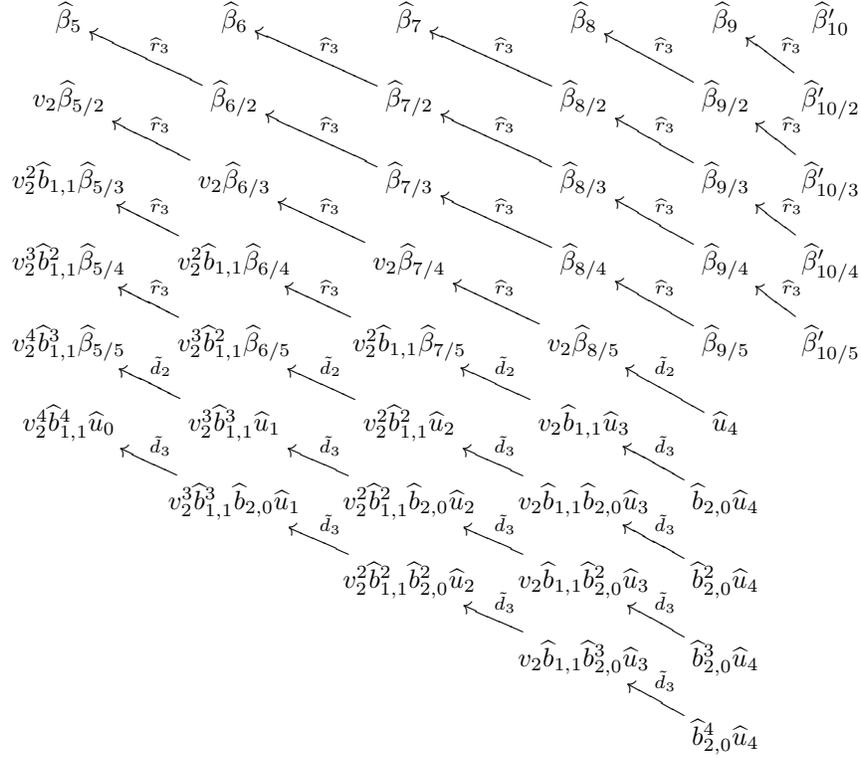

\begin{proof}[Proof of \autoref{thm-7.3.15} \eqref{thm-7.3.15-3}]
For $k=0$ this is a direct consequence of \eqref{thm-7.3.15-1} via multiplication by $\hh_{1,1}$.
We will illustrate with the case $i=p-1$ and $k \le 2$, 
and the other cases are similarly shown.  
For $k=1$, we have the sequence analogous to that of \autoref{quillen-massey-rel}: 
\[
\bb_{2,0}\uu_{p-1} \stackrel{\tilde{d}_{3}}{\longrightarrow}
v_{2}\bb_{1,1}\uu_{p-2} \stackrel{\tilde{d}_{2}}{\longrightarrow}
v_{2}^{2}\bb_{1,1}\bbeta_{2p-3/p} \stackrel{\wrr_{p}}{\longrightarrow}
\cdots \stackrel{\wrr_{p}}{\longrightarrow}
v_{2}^{2}\bb_{1,1}\bbeta_{p/3}.
\]
This allows us to identify 
$v_{2}\hh_{1,1}\bb_{1,1}\uu_{p-2}$ 
with the Massey product 
$\mu_{p-1}(v_{2}^{2}\bb_{1,1}\bbeta_{p/3})$
up to unit scalar multiplication.
It follows that the differential on 
$\hh_{2,0}\hh_{1,1}(\bb_{2,0}\uu_{p-1})$
is the value of 
$\hh_{2,0}\hh_{1,1} \mu_{p-1}(v_{2}^{2}\bb_{1,1}\bbeta_{p/3})$.
Now $\hh_{2,0}\hh_{1,1}$ (resp.~$\bb_{1,1}$) is the image of $\bbeta_{2}$ (resp.~$\bbeta_{p/p}$) 
under a suitable reduction map, so we have 
\begin{equation*}
\begin{split}
\tilde{d}_{5}(\hh_{1,1}\hh_{2,0}\bb_{2,0}\uu_{p-1})
 & = \hh_{2,0}\hh_{1,1}\mu_{p-1}(v_{2}^{2}\bb_{1,1}\bbeta_{p/3})
   = v_{2}^{2}\bb_{1,1}\bbeta_{2} \mu_{p-1}(\bbeta_{p/3})  \\
 & = v_{2}^{2}\bb_{1,1} \mu_{p-1}(\bbeta_{2}) \bbeta_{p/3}
\qquad
\mbox{by \autoref{lem-definingmu} }  \\
 & = v_{2}^{2}\bb_{1,1} \mu_{p-1}(\bbeta_{2}) v_{1}^{p-3}\bbeta_{p/p} 
   = v_{2}^{2} \bb_{1,1}^{2} v_{1}^{p-3}\mu_{p-1}(\bbeta_{2})  \\
 & = v_{2}^{2} \bb_{1,1}^{2} \mu_{2}(\bbeta_{p-1})
\qquad
\mbox{by \autoref{prop-2massey} } \\
 & = v_{2}^{2} \bb_{1,1}^{2} \hh_{1,1}\bbeta'_{p/2}
\end{split}
\end{equation*}
as claimed. 
For $k=2$, 
we have the sequence\footnote{Note that we may assume that $p \ge 5$ since $0 \le k <p-1$.}
\[
\bb_{2,0}^{2}\uu_{p-1} \stackrel{\tilde{d}_{3}}{\longrightarrow}
v_{2}\bb_{1,1}\bb_{2,0}\uu_{p-2} \stackrel{\tilde{d}_{3}}{\longrightarrow}
v_{2}^{2}\bb_{1,1}^{2}\uu_{p-3} \stackrel{\tilde{d}_{2}}{\longrightarrow}
v_{2}^{3}\bb_{1,1}^{2}\bbeta_{2p-4/p} \stackrel{\wrr_{p}}{\longrightarrow}
\cdots \stackrel{\wrr_{p}}{\longrightarrow}
v_{2}^{3}\bb_{1,1}^{2}\bbeta_{p/4}
\]
By the similar argument to the case $k=1$, we have 
\begin{equation*}
\begin{split}
\tilde{d}_{7}(\hh_{1,1}\hh_{2,0}\bb_{2,0}^{2}\uu_{p-1}) 
 & = \hh_{2,0}\hh_{1,1}\mu_{p-1}(v_{2}^{3}\bb_{1,1}^{2}\bbeta_{p/4})
= v_{2}^{3}\bb_{1,1}^{2}\bbeta_{2}\mu_{p-1}(\bbeta_{p/4})  \\
 & = v_{2}^{3}\bb_{1,1}^{2}\mu'_{p-1}(\bbeta_{2})\bbeta_{p/4}
 \qquad
\mbox{by \autoref{lem-definingmu} }  \\
 & = v_{2}^{3}\bb_{1,1}^{2}\mu'_{p-1}(\bbeta_{2})v_{1}^{p-4}\bbeta_{p/p}
 = v_{2}^{3}\bb_{1,1}^{3} v_{1}^{p-4}\mu'_{p-1}(\bbeta_{2}) \\
 & = v_{2}^{3}\bb_{1,1}^{3}\mu'_{3} (\bbeta_{p-2}) 
\qquad
\mbox{by \autoref{prop-2massey} }  \\
& = v_{2}^{3}\bb_{1,1}^{3}\hh_{1,1}\bbeta'_{p/3}
\end{split}
\end{equation*}
as claimed. 
\end{proof}

\begin{proof}[Proof of \autoref{thm-7.3.15} \eqref{thm-7.3.15-4} and \eqref{thm-7.3.15-5}]
We have the sequence 
\[
\bb_{2,0}^{k}\uu_{p-1} \stackrel{\tilde{d}_{3}}{\longrightarrow}
\cdots \stackrel{\tilde{d}_{3}}{\longrightarrow}
v_{2}^{k}\bb_{1,1}^{k}\uu_{p-1-k} \stackrel{\tilde{d}_{2}}{\longrightarrow}
v_{2}^{k+1}\bb_{1,1}^{k}\bbeta_{2p-2-k/p} \stackrel{\wrr_{p}}{\longrightarrow}
\cdots \stackrel{\wrr_{p}}{\longrightarrow}
v_{2}^{k+1}\bb_{1,1}^{k}\bbeta_{p/k+2}
\]
for $1 \le k<p-1$, and 
\[
\bb_{2,0}^{p-1}\uu_{p-1} \stackrel{\tilde{d}_{3}}{\longrightarrow}
\cdots \stackrel{\tilde{d}_{3}}{\longrightarrow}
v_{2}^{p-1}\bb_{1,1}^{p-1}\uu_{0}
\]
for $k=p-1$. 
Thus we have 
\[
\tilde{d}_{r} (\bb_{2,0}^{k}\uu_{p-1}) = 
\begin{cases}
\mu_{p-1} (v_{2}^{k+1}\bb_{1,1}^{k}\bbeta_{p/k+2}) 
& \mbox{for }k<p-1  \\
\mu_{p-1} (v_{2}^{p-1}\bb_{1,1}^{p-1}\uu_{0}) 
& \mbox{for }k=p-1
\end{cases}
\]
up to unit scalar multiplication. 
Since $\hh_{1,1}\mu_{p-1} (x)=\bb_{1,1}x$ we have 
\[
\tilde{d}_{r}(\hh_{1,1}\bb_{2,0}^{k}\uu_{p-1}) = 
\begin{cases}
v_{2}^{k+1}\bb_{1,1}^{k+1}\bbeta_{p/k+2} 
& \mbox{for }k<p-1  \\
v_{2}^{p-1}\bb_{1,1}^{p}\uu_{0} 
& \mbox{for }k=p-1
\end{cases}
\]
as claimed.
\end{proof}

\appendix

\section{Massey products}\label{app-massey}

Here we recall the definition and properties of Massey products very briefly
(cf.~\cite[A1.4]{Rav:MU2nd}) and prove some results used in this paper.
Let $C$ be a {\DGA}, which makes  $H^{*} (C)$ a graded algebra.  
For $x \in C$ or $x \in H^{*}(C)$, let $\overline{x}=(-1)^{1+\deg(x)}x$,
where $\deg(x)$ denotes the total degree:  
the sum of its internal and cohomogical degrees of $x$.  
Then we have $d (\overline{x})=-\overline{d (x)}$,
$\overline{(xy)}=-\overline{x}\,\overline{y}$, and $d (xy)=d
(x)y-\overline{x}d (y) $.

Let $\alpha_{k} \in H^{*} (C)$ $(k=1,2,\dotsc)$ be a finite collection
of elements and with representative cocycles $a_{k-1,k} \in C$.  When
$\overline{\alpha_{1}}\alpha_{2}=0$ and
$\overline{\alpha_{2}}\alpha_{3}=0$, there are cochains $a_{0,2}$ and
$a_{1,3}$ such that $d(a_{0,2}) = \overline{a_{0,1}}a_{1,2}$ and
$d(a_{1,3}) = \overline{a_{1,2}}a_{2,3}$, and we have a cocycle
$b_{0,3} = \overline{a_{0,2}}a_{2,3}+\overline{a_{0,1}}a_{1,3}$.  
The corresponding class in $H^{*}(C)$ represents the Massey product
$\langle \alpha_{1}, \alpha_{2}, \alpha_{3} \rangle$, 
which is the coset comprising all cohomology classes 
represented by such $b_{0,3}$
for all possible choices of $a_{i,j}$.  Two choices of $a_{0,2}$ or
$a_{1,3}$ differ by a cocycle.  The {\bf indeterminacy} of $\langle
\alpha_{1}, \alpha_{2}, \alpha_{3} \rangle$ is the set
\[
\alpha_{1}H^{|\alpha_{2}\alpha_{3}|} (C)
        +H^{|\alpha_{1}\alpha_{2}|} (C)\alpha_{3}.
\]
If the triple product contains zero, then one
such choice yields a $b_{0,3}$ which is the coboundary of a cochain
$a_{0,3}$.

If we have two 3-fold Massey products 
$\langle \alpha_{1}, \alpha_{2}, \alpha_{3} \rangle$ 
and 
$\langle \alpha_{2}, \alpha_{3}, \alpha_{4} \rangle$ 
containing zero, then the $a_{i-1,i}$ and
$a_{i-2,i}$ can be chosen so that there are cochains $a_{0,3}$ and
$a_{1,4}$ with $d(a_{0,3}) = b_{0,3}$ and $d(a_{1,4}) = b_{1,4}$, and
the 4-fold Massey product 
$\langle \alpha_{1}, \alpha_{2}, \alpha_{3}, \alpha_{4} \rangle$ 
represented by the cocycle
$b_{0,4} = \overline{a_{0,3}}a_{3,4} + \overline{a_{0,2}}a_{2,4} +
\overline{a_{0,1}}a_{1,4}$.  More generally, 
if we have cocycles $b_{j,k}$ and cochains $a_{j,k}$ satisfying
\setcounter{equation}{\value{thm}}
\begin{equation}\label{elements-masseyprod}
b_{j,k} = \sum_{j<\ell<k} \overline{a_{j,\ell}}a_{\ell,k} 
\quad
\mbox{for $i \le j < k \le i+n$}
\end{equation}
\setcounter{thm}{\value{equation}}%
and $d(a_{j,k}) = b_{j,k}$ for $0<k-j<n$, then we have the
$n$-fold Massey products $\langle \alpha_{i+1}, \ldots, \alpha_{i+n}
\rangle$ represented by $b_{i,i+n}$.  The cochains $a_{j,k}$ chosen
above are called the {\bf defining system} for the Massey product. 

\bigskip

If two products 
$\langle \alpha_{1}, \dotsc, \alpha_{n-1} \rangle$
and 
$\langle \alpha_{2}, \dotsc, \alpha_{n} \rangle$ 
are strictly defined 
(meaning all the lower order products in sight have trivial indeterminacy), 
then we have 
\[
\alpha_{1}\langle \alpha_{2}, \dotsc, \alpha_{n} \rangle =
\langle \overline{\alpha}_{1}, \dotsc, \overline{\alpha_{n-1}} \rangle \alpha_{n}. 
\]
In fact, we can relax the hypothesis of strict definition 
in the following way.

\begin{lem}\label{lem-juggling}
Suppose that 
$\langle \alpha_{1}, \dotsc, \alpha_{n-1} \rangle$
and 
$\langle \alpha_{2}, \dotsc, \alpha_{n} \rangle$
are defined and have representatives $x$ and $y$
respectively with the common defining system $a_{i,j}$ $(0<i<j<n)$.  
Then, the cocycle $\overline{x} a_{n-1,n}$ is cohomologous to $a_{0,1}y$.
\end{lem}

\begin{proof}
If both $x$ and $y$ contain zero, then we would have cochains $a_{1,n}$
and $a_{0,n-1}$ satisfying $d(a_{0,n-1})=x$ and $d(a_{1,n})=y$. Hence we
could define the cocycle $b_{0,n}$ \eqref{elements-masseyprod}.  In
that case we would have
\begin{align*}
d(b_{0,n})
& = d(\overline{a_{0,1}}a_{1,n}) + d(\overline{a_{0,n-1}}a_{n-1,n}) 
+ d(\tilde{b}_{0,n})  \\
& = -a_{0,1}y + \overline{x}a_{n-1,n}
+ d(\tilde{b}_{0,n})
= 0
\end{align*}
where 
\[
\tilde{b}_{0,n} = \sum_{1<i<n-1}\overline{a_{0,i}}a_{i,n}. 
\]
Even if $x$ and $y$ do not contain zero, so we don't have cochains
$a_{1,n}$ and $a_{0,n-1}$, we can still define $\tilde{b}_{0,n}$.
A routine calculation gives the desired value of  $d(\tilde{b}_{0,n})$.
\end{proof}

We also have Massey products in the {\SS} associated with a filtered {\DGA}
or a filtered differential graded module over a filtered {\DGA}.  
Though our {\CESS} is not associated with such a filtration, 
we can get around this as follows.  
Let
$T_{m}^{*} = \bigoplus_{i \ge 0}T_{m}^{i}$
be a bigraded comodule algebra with $i$ being the second grading
and the algebra structure given by the pairings
$T_{m}^{i}\otimes_{} T_{m}^{j} \to T_{m}^{i+j}$.

Recall that for a Hopf algebroid $(A,\Gamma)$ and a comodule algebra $M$ 
the cup product in the cobar complex $C = C_{\Gamma}(M)$ is given by 
\begin{align*}
\lefteqn{\left(\gamma_{1}\otimes_{}\dotsb\otimes_{}\gamma_{s}\otimes_{}m_{1}
\right)\cup
\left(\gamma_{s+1}\otimes_{}\dotsb\otimes_{}\gamma_{s+t}\otimes_{}m_{2}\right)}
\qquad\qquad\\
& = 
\gamma_{1}\otimes_{}\dotsb\otimes_{}\gamma_{s}
\otimes_{}m_{1}^{(1)}\gamma_{s+1}\otimes_{}\dotsb \otimes_{}
m_{1}^{(t)}\gamma_{s+t}\otimes_{} m_{1}^{(t+1)}m_{2} 
\end{align*}
where $\gamma_{i}\in \Gamma (m+1)$ and $m_{j} \in M$, 
and 
$m_{1}^{(1)}\otimes_{}\dotsb \otimes_{} m_{1}^{(t+1)}$
is the iterated coproduct on $m_{1}$.  
The coboundary operator is a derivation with respect to this product
and $C$ is a filtered {\DGA}; 
we have 
\[
d(x \cup y) = d(x) \cup y + (-1)^{\deg(x)}x \cup d(y). 
\]

Now we have consider the two quadrigraded {\CESS}s:
\setcounter{equation}{\value{thm}}
\begin{equation}\label{eq-WCESS-1}
\Ext_{G(m+1)}(\Ext_{\Gamma(m+2)}(T_{m}^{*}))
\implies \Ext_{\Gamma (m+1)}(T_{m}^{*}),
\end{equation}
\setcounter{thm}{\value{equation}}%
which is associated with a filtration on $C=C_{\Gamma
(m+1)}(T_{m}^{*})$, and \setcounter{equation}{\value{thm}}
\begin{equation}\label{eq-WCESS-2}
\Ext_{G(m+1)}(\Ext_{\Gamma(m+2)}(T_{m}^{*} \otimes E_{m+1}^{1}))
   \implies \Ext_{\Gamma(m+1)}(T_{m}^{*} \otimes E_{m+1}^{1}),
\end{equation}
\setcounter{thm}{\value{equation}}%
which is associated with a filtration on $C'=C_{\Gamma
(m+1)}(T_{m}^{*} \otimes E_{m+1}^{1})$.  
We may regard the {\CESS} of
\eqref{CESS-for-E_{m+1}^{1}} as a quotient of the degree $p^{i}-1$
component of \eqref{eq-WCESS-2}.

\bigskip

Since $C'$ is a left differential module over $C$, 
\eqref{eq-WCESS-2} is a module over \eqref{eq-WCESS-1}. 
Then we can make a similar product  
$\langle \alpha_{1}, \dotsc, \alpha_{j} \rangle$
with $\alpha_{i} \in H^{*} (C)$ $(1 \le i<j)$ 
and $\alpha_{j} \in H^{*} (C')$ 
under certain conditions. 
In particular, we will be interested in Massey products of the form
\setcounter{equation}{\value{thm}}
\begin{equation}\label{eq-muky}
\mu_{k}(y) = \langle \hh_{1,1}, \dotsc, \hh_{1,1}, y\rangle 
\qquad
\mbox{and}
\qquad 
\mu'_{k}(x) = \langle x, \hh_{1,1}, \dotsc, \hh_{1,1} \rangle 
\end{equation}
\setcounter{thm}{\value{equation}}%
with $k$ factors $\hh_{1,1}$.  
For $1<k<p$, $\mu_{k}(y)$ is defined only if $0 \in \mu_{k-1}(y)$.  
If $\mu_{k}(\mu_{p-k}(y))$ is defined
for some $k$, then it contains $\bb_{1,1}y$.

\begin{rem}
$\hh_{1,1} \in \Ext_{\Gamma (m+1)}^{1}(T_{m}^{p-1})$ is represented in the
cobar complex by
\[
x = -d (\wt_{1}^{p})
= \left(\wt_{1}\otimes_{}1+1\otimes_{}\wt_{1} \right)^{p}
        -1\otimes_{}\wt_{1}^{p} \\
\equiv \wt_{1}^{p}\otimes_{}1 \qquad \bmod (p), 
\]
which means that $\hh_{1,1}$ becomes trivial 
when we pass to $\Ext_{\Gamma (m+1)}^{1}(T_{m}^{p})$.
Similarly, we have
\[
x \cup x 
= d\left( x \cup \wt_{1}^{p} \right)
= d\left( \sum_{i>0}\binom{p}{i} \wt_{1}^{i}\otimes_{}\wt_{1}^{2p-i} \right).
\]
Thus 
$\hh_{1,1}\cup \hh_{1,1} \in \Ext_{\Gamma(m+1)}^{2}(T_{m}^{2p-2})$ 
maps trivially to 
$\Ext_{\Gamma(m+1)}^{2}(T_{m}^{2p-1})$.
\end{rem}

\begin{lem}\label{lem-xi}
Let $x_{1}=x$ as above and define $x_{i}$ inductively on $i$ by 
\[
x_{i}=\left(x_{i-1}\cup \wt_{1}^{p}- \wt_{1}^{p}\cup x_{i-1} \right)/i 
\qquad
(1<i<p).
\]
Then $x_{i}$ is in $C_{\Gamma (m+1)} (T_{m}^{(i-1)(p-1)})$ and it satisfies 
\[
x_{i} \equiv (-1)^{i+1}\wt_{1}^{ip}\otimes_{}1/i! \quad \bmod (p) 
\qquad
\mbox{and}
\qquad
d(x_{i}) = \sum_{0<j<i}x_{j}\cup x_{i-j}. 
\]
\end{lem}

\begin{proof}
We prove these statements by induction. 
For the first statement, assume that 
$x_{i} \in C_{\Gamma (m+1)}(T_{m}^{(i-1)(p-1)})$. 
This means that it has the form
$c\wt_{1}^{i+p-1}\otimes_{}\wt_{1}^{(i-1)(p-1)}$
modulo $C_{\Gamma (m+1)}(T_{m}^{(i-1) (p-1)-1})$ 
for some scalar $c$, 
and so we have
\begin{align*}
x_{i+1} 
& = (x_{i} \cup \wt_{1}^{p} - \wt_{1}^{p}\cup x_{i} )/(i+1)  \\
& \equiv c (\wt_{1}^{i+p-1}\otimes_{}\wt_{1}^{(i-1)(p-1)+p} - \wt_{1}^{i+p-1}\otimes_{}\wt_{1}^{(i-1) (p-1)+p} ) /(i+1)
\equiv 0 
\end{align*}
modulo $C_{\Gamma (m+1)} (T_{m}^{i(p-1)})$. 
For the congruence, we see that 
\begin{align*}
(i+1)!x_{i+1} 
& = 
i!(x_{i}\cup \wt_{1}^{p}-\wt_{1}^{p}\cup x_{i})
\equiv 
(-1)^{i+1}\left((\wt_{1}^{ip}\otimes 1 )\cup \wt_{1}^{p}
-  \wt_{1}^{p} \cup (\wt_{1}^{ip}\otimes 1 ) \right)  \\
& = 
(-1)^{i+1}\left(\wt_{1}^{ip} \otimes \wt_{1}^{p}
-  \wt_{1}^{(i+1)p}\otimes 1  
-  \wt_{1}^{ip}\otimes \wt_{1}^{p} \right)
= 
(-1)^{i+2}\wt_{1}^{(i+1)p}\otimes 1. 
\end{align*}
For the derivation formula, we see that 
\begin{align*}
\lefteqn{(i+1)d(x_{i+1}) - x_{i}\cup x_{1} - x_{1}\cup x_{i}} 
\qquad\qquad  \\
& = 
d (x_{i})\cup \wt_{1}^{p} 
- \wt_{1}^{p}\cup d(x_{i})  \\
& = 
\left(\sum_{0<j<i} x_{j}\cup x_{i-j} \right)\cup \wt_{1}^{p}
- \wt_{1}^{p}\cup \left(\sum_{0<j<i} x_{j}\cup x_{i-j} \right)  \\
& = \sum_{0<j<i}
\Big( 
x_{j} \cup (x_{i-j} \cup \wt_{1}^{p} - \wt_{1}^{p} \cup  x_{i-j})
+ (x_{j} \cup  \wt_{1}^{p} - \wt_{1}^{p} \cup x_{j}) \cup x_{i-j}
\Big)  \\
& = \sum_{0<j<i}
\Big( 
(i+1-j) x_{j}\cup x_{i+1-j}
+ (j+1)x_{j+1} \cup x_{i-j}
\Big)  \\
& = 
(i+1)\sum_{1<j<i}x_{j}\cup x_{i+1-j}. 
\qedhere
\end{align*}
\end{proof}

The following result follows easily from \autoref{lem-xi}.

\begin{lem}\label{lem-definingmu}
Suppose that
$\alpha, \beta \in \Ext_{\Gamma (m+1)}(T_{m}^{h}\otimes_{} E_{m+1}^{2})$
are represented by cocycles $a_{1}$ and $b_{1}$, and that there are cochains
\[
a_{i}, b_{i}\in C_{\Gamma (m+1)}(T_{m}^{h+(i-1) (p-1)}\otimes_{}E_{m+1}^{2})
\quad
\mbox{for
$1<i\le k$}
\]
satisfying 
\[
d(a_{i}) = \sum_{0<j<i} a_{i-j}\cup x_{j}
\quad
\mbox{and}
\quad 
d(b_{i}) = \sum_{0<j<i} x_{j}\cup b_{i-j}, 
\]
where $x_{j}$ are as in \autoref{lem-xi}.
Then the Massey products
\[
\mu'_{k} (\alpha ), \mu_{k} (\beta ) \in 
\Ext_{\Gamma (m+1)}(T_{m}^{h+k (p-1)}\otimes_{} E_{m+1}^{2})
\]
are defined and are represented by the cocycles
\[
\sum_{0<i<k+1}a_{k+1-i} \cup x_{i} 
\quad
\mbox{and}
\quad 
\sum_{0<i<k+1}x_{i} \cup b_{k+1-i}.
\]
Moreover, we have 
$\alpha \mu_{k} (\beta ) = \mu '_{k} (\overline{\alpha} )\beta$
using these representatives.
\end{lem}

Here are two examples of such products.

\begin{ex}\label{prop-2massey}
For $0<k<p$ and $\ell>0$, 
the Massey product $\mu_{k}(\bbeta'_{p\ell-k+1})$ is defined and it is represented by
\[
\sum_{0<i<k+1} x_{i} \cup 
(-1)^{k-i}\frac{(p\ell-k)!}{(p\ell-i)!}\bbeta'_{p\ell+1-i/k+1-i}.
\]
We have an equality 
$v_{1}\mu_{k} (\bbeta '_{p\ell +1-k}) = \mu_{k-1} (\bbeta '_{p\ell +2-k})/ (k-1-p\ell )$ 
for $k>1$.
\end{ex}

\begin{ex}\label{prop-3massey}
For $0<k<p$ and $\ell>0$, 
the Massey product
$\mu_{k}(\bbeta_{p\ell/p+2-k})$ is defined and it is represented by
\[
x_{1}\cup v_{2}^{-1}\uu_{p\ell+k-1-p}
+\sum_{1<i<k+1} x_{i} \cup 
(-1)^{i+1}\frac{(p\ell +k)!}{\ell (p\ell +k-i)!}
\bbeta_{p\ell+k-i/p+2-i}.
\]
\end{ex}

\end{document}